\documentclass[preprint,3p]{elsarticle}
\usepackage[toc,page,title,titletoc,header]{appendix}
\usepackage{stmaryrd}
\SetSymbolFont{stmry}{bold}{U}{stmry}{m}{n}
\usepackage{amsmath}
\usepackage{amssymb}
\usepackage{algorithm}
\usepackage{amsthm}
\usepackage{tabularx}
\usepackage{subfigure}
\usepackage{epsfig}
\usepackage{indentfirst}
\usepackage{cases}
\biboptions{sort&compress}
\usepackage{graphicx}
\usepackage{epstopdf}
\usepackage{color}
\usepackage{float}
\usepackage{algorithmic}
\usepackage{bm}
\usepackage{subfig}
\usepackage{multirow}
\usepackage{comment}
\usepackage{booktabs}
\usepackage{adjustbox}

\renewcommand{\vec}[1]{\mathbf{#1}}

\newcommand{\be}{\begin{equation}}
\newcommand{\ee}{\end{equation}}
\newcommand{\ba}{\begin{array}}
\newcommand{\ea}{\end{array}}
\newcommand{\bea}{\begin{eqnarray}}
\newcommand{\eea}{\end{eqnarray}}
\newcommand{\beas}{\begin{eqnarray*}}
\newcommand{\eeas}{\end{eqnarray*}}

\newtheorem{thm}{Theorem}[section]
\newtheorem{prop}{Proposition}[section]

\numberwithin{equation}{section}

\begin{document}

\begin{frontmatter}

\title{High-order temporal parametric finite element methods for simulating solid-state dewetting}

\author[1]{Xiaowen Gan}
\address[1]{School of Mathematics and Statistics, Wuhan University, Wuhan 430072, China}
\ead{ganxiaowen@whu.edu.cn}

\author[1]{Yuqian Teng}
\ead{t-yuqian@163.com}

\author[1]{Sisheng Wang}
\ead{2015301000004@whu.edu.cn}




\begin{abstract}
We propose a class of temporally high-order parametric finite element methods for simulating solid-state dewetting of thin films in two dimensions using a sharp-interface model. The process is governed by surface diffusion and contact point migration, along with appropriate boundary conditions. By incorporating the predictor-corrector strategy and the backward differentiation formula for time discretization into the energy-stable parametric finite element method developed by Zhao et al. (2021), we successfully construct temporally high-order schemes. The resulting numerical scheme is semi-implicit, requiring the solution of a linear system at each time step. The well-posedness of the fully discretized system is established. Moreover, the method maintains the long-term mesh equidistribution property. 
Extensive numerical experiments demonstrate that our methods achieve the desired temporal accuracy, measured by the manifold distance, while maintaining good mesh quality throughout the evolution.
\end{abstract}




\end{frontmatter}
\section{Introduction}

Due to surface tension and capillarity effects (especially at high temperatures and/or over long time scales), solid-state dewetting is a ubiquitous phenomenon in thin film technology \cite{Jiran1990,Ye2010,Leroy2012,Thompson2012}. A solid thin film deposited on a substrate can break up and agglomerate into isolated islands when heated to a critical temperature, which is well below the material’s melting point. This process is referred to as solid-state dewetting because it occurs while the film remains in the solid state. Solid-state dewetting has been observed and studied in a wide range of material systems \cite{Rabkin2014,Armelao2006,Rath2007,Randolph2007,Schmidt2009}. For a more comprehensive introduction to this field, interested readers are referred to the review by \cite{Thompson2012,Leroy2016}.

\begin{figure}[h!]
\hspace{-7mm}
\centering
\includegraphics[width=4.9in,height=2.5in]{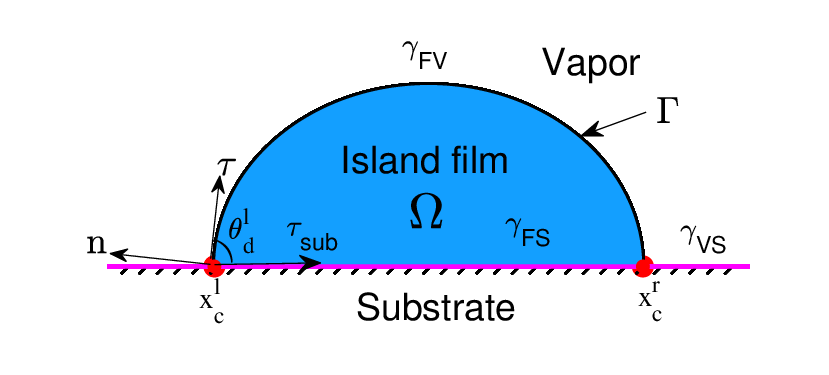}
\vspace{-10mm}
\caption{ A schematic illustration of an island film on a flat, rigid substrate in two dimensions with three interaces, i.e., island film/vapor (FV), island film/substrate (FS), and vapor/substrate (VS) interface, where
$ x_c^l$ and $x_c^r$ denote the $x$-cordinates of the  left and right contact points, $\gamma_{FV}$ , $\gamma_{VS}$ and $\gamma_{FS}$ represent the film/vapor, vapor/substrate and film/substrate surface energy densities, respectively, 
$\bm{\tau}=(\partial_sx,\partial_sy)^{T}$ and $\bm{n}=(-\partial_sy,\partial_sx)^{T}$ are the unite tangential and  normal vectors, and $\bm{\tau}_{sub}=(1,0)^{T}$ represents the unit vector along the x-coordinate.}
\label{fig:1}
\end{figure}

Numerous numerical methods have been proposed for simulating solid-state dewetting. These include parametric methods \cite{Wang2015,Jiang2016,Bao2017-dewetting,Jiang2019,Li2026,Zhao2020,Zhao2021}, phase-field methods \cite{Jiang2012,Dziwnik2017,Huang2019}, level set methods \cite{Chopp1999,Smereka2003,Kolahdouz2013}, marker-particle methods \cite{Wong2000,Du2010}, crystalline formulation methods \cite{Carter1995,Zucker2013}, discrete surface chemical potential methods \cite{Dornel2006}, and kinetic Monte Carlo methods \cite{Pierre2009}.
Among these, the parametric approach based on a sharp-interface model can accurately capture the physics of surface diffusion and contact line migration in solid-state dewetting. It is also computationally efficient and straightforward to implement. The parametric finite element method (PFEM) has emerged as a prevalent technique for solving geometric flows \cite{Deckelnick-Dziuk2005,Barrett2020}. The method was originally introduced by Dziuk  \cite{Dziuk1991} for surfaces in $\mathbb{R}^3$ and subsequently extended to curves in $\mathbb{R}^2$ in  \cite{Dziuk1994}. In 2007, Barrett et al. introduced the BGN scheme, a novel parametric finite element method for surface diffusion flow of closed curves \cite{BGN2007}, which employs piecewise linear finite element spatial discretization and a backward Euler temporal discretization. This scheme avoids remeshing and maintains high mesh quality during evolution by incorporating implicit tangential motion. Consequently, the BGN scheme has been widely applied to problems such as the Mullins–Sekerka problem \cite{BGN2010-1}, the Stefan problem \cite{BGN2010-2,BGN2013-1}, two-phase flow \cite{BGN2013-2,BGN2015}, and biofilm growth dynamics \cite{BGN2016}, as well as to modeling solid-state dewetting processes \cite{Bao2017-dewetting,Li2025,Zhao2020,Zhao2021}.

In general, the solid-state dewetting problem can be regarded as a type of geometric evolution problem for open curves or surfaces, governed by surface diffusion and contact point (in 2D) or contact line (in 3D) migration. This is often referred to as a moving contact line problem \cite{Wang2015,Bao2017-dewetting, Zhao2020}.
As illustrated in Figure \ref{fig:1}, consider an open curve $\Gamma(t) = \vec{X}(s,t) = (x(s,t), y(s,t))^T$ in two dimensions. Its two contact points, defined as $(x_c^l(t),0) = (x(0,t),0)$ and $(x_c^r(t),0) = (x(L(t),t),0)$, migrate along the substrate (i.e., x-axis). Here, the arc length parameter $s$ ranges from $0$ to $L(t)$, where $L(t) = |\Gamma(t)|$ denotes the total length of the curve at time $t$.
The total interfacial energy of the system is defined as $W(t) = \int_{\Gamma(t)} 1 ds - \sigma [x_c^r(t) - x_c^l(t)]$. The dimensionless material parameter $\sigma$ is given by $\sigma = (\gamma_{VS} - \gamma_{FS}) / \gamma_{FV}$, where $\gamma_{FV}$, $\gamma_{VS}$, and $\gamma_{FS}$ represent the film-vapor, vapor-substrate, and film-substrate interfacial energy densities, respectively.
Through thermodynamic variations of the total interfacial energy, the following sharp-interface model has been derived to simulate solid-state dewetting of thin films with isotropic surface energy in two dimensions \cite{Wang2015,Bao2017,Bao2017-dewetting,Zhao2021}:

\begin{subequations}\label{eqn:2diso}
\begin{align}
\partial_{t}\vec{X}&=\partial_{ss}\kappa \; \vec{n}, \quad 0<s<L(t),\quad t>0, \label{eqn:2diso1}\\
\kappa&=-\left(\partial_{ss}\vec{X}\right)\cdot\vec{n},  \label{eqn:2diso2}
\end{align}
\end{subequations}
where $\vec n=(-\partial_s y, \partial_s x)^T$ is the unit normal vector and $\kappa$ is the curvature of the curve.
The boundary conditions are given as
\begin{itemize}
	\item [(i)] contact point condition
	\begin{equation}
	\label{eqn:weakBC1}
	y(0,t)=0,\quad y(L(t),t) = 0,\quad t\geq0;
	\end{equation}
	\item [(ii)] relaxed contact angle condition
	\begin{equation}
	\label{eqn:weakBC2}
	\frac{\mathrm{d} x_c^l(t)}{\mathrm{d t}}=\eta (\cos\theta_d^l - \sigma),\qquad
	\frac{\mathrm{d} x_c^r(t)}{\mathrm{d t}}=-\eta (\cos\theta_d^r -\sigma),\qquad t\ge0;
	\end{equation}
	\item [(iii)]zero-mass flux condition
	\begin{equation}
	\label{eqn:weakBC3}
	\partial_s\kappa(0,t) =0,\qquad\partial_s\kappa(L(t),t) = 0, \qquad t\geq 0;
	\end{equation}
\end{itemize}
where $\theta_d^l=\theta_d^l(t)$ and $\theta_d^r=\theta_d^r(t)$ represent the dynamic contact angles at the left and right moving contact points, respectively, and $\eta>0$ denotes the contact line mobility. 
This parameter governs the relaxation rate at which the dynamic contact angles $\theta_d^l$ and $\theta_d^r$ approach the equilibrium contact angle $\theta_i$. The equilibrium contact angle $\theta_i \in (0, \pi)$, also known as the isotropic Young's angle, satisfies the Young equation: $\cos\theta_i = \sigma = (\gamma_{\text{VS}} - \gamma_{\text{FS}}) / \gamma_{\text{FV}}$.
The initial curve is given by:
\begin{equation}
\vec{X}(0)=\vec{X}(s,0)=(x(s,0), y(s,0))^T, \qquad 0 \le s \le L(0),
\label{eqn:initial}
\end{equation}
with $x_c^l(0) < x_c^r(0)$. Here, $\Omega$ denotes the island film (see Figure \ref{fig:1}), and $|\Omega| = \int_{0}^{L(t)} y  \partial_s x ds$ represents the area (or mass) of the film. This area is defined for the region enclosed by the open curve $\Gamma(t)$ and the substrate. By combining \eqref{eqn:2diso} with the boundary conditions, one can prove that \cite{Bao2017}:
 $$\frac{\mathrm{d}}{\mathrm{d t}}|\Omega|=0,\quad \frac{\mathrm{d}}{\mathrm{d } t}W=-\int_{0}^{L(t)}(\partial_s\kappa)^2ds-\frac{1}{\eta}\left[\Big(\frac{\mathrm{d} x_c^l(t)}{\mathrm{d } t}\Big)^2+\Big(\frac{\mathrm{d} x_c^r(t)}{\mathrm{d } t}\Big)^2\right]\leq0,\quad\forall~t\geq0.$$
This implies conservation of mass/area and dissipation of energy over time.

In 2021, Zhao, Jiang and Bao proposed an energy-stable parametric finite element method \cite{Zhao2021} by reformulating the relaxed contact angle condition \eqref{eqn:weakBC2} into a Robin-type boundary condition as:
\begin{equation}\label{eq:Robin}
\partial_{s}x|_{s=0}=\sigma+\frac{1}{\eta}\frac{\mathrm{d} x_c^l(t)}{\mathrm{d } t},\quad \partial_{s}x|_{s=L(t)}=\sigma-\frac{1}{\eta}\frac{\mathrm{d} x_c^r(t)}{\mathrm{d } t}.
\end{equation}
They obtained a new variational formulation by incorporating the boundary conditions \eqref{eqn:weakBC3} and \eqref{eq:Robin} through integration by parts. This approach simultaneously handles the motion of the interface curve and its contact points, and energy dissipation is proven for the variational formulation, spatial semi-discretization, and full discretization. Furthermore, the numerical results demonstrate second-order accuracy in space and well-maintained mesh quality. We refer to this method as the ZJB scheme.

The time discretization of the ZJB scheme utilizes the backward Euler method, resulting in first-order accuracy in time. To improve the accuracy and efficiency of the scheme, it is desirable to develop high-order temporal schemes for simulating solid-state dewetting problems. Recently, Jiang, Su, and Zhang developed higher-order temporal schemes for surface diffusion flow based on the BGN method \cite{Jiang2024,Jiang-BDF2024,Jiang2025}. The high-order temporal accuracy of these schemes has been verified numerically using the manifold distance.

The main aim of this paper is to develop and investigate high-order temporal schemes for simulating solid-state dewetting problems in 2D. Inspired by \cite{Zhao2021,Jiang-BDF2024,Jiang2025}, we incorporate the predictor-corrector method and backward differentiation formula (BDF) method for time discretization into the ZJB scheme. The proposed high-order temporal scheme inherits the advantageous properties of the ZJB scheme, such as well-posedness and long-term mesh equidistribution.

The remainder of this paper is organized as follows. In Section 2, we provide a brief overview of the ZJB scheme. In Section 3, we propose our high-order temporal schemes. Extensive numerical experiments are performed in Section 4 to demonstrate the high-order temporal accuracy and efficiency of the proposed schemes. Finally, we draw conclusions in Section 5.

\section{Review of the ZJB scheme.}
In this section, we review the ZJB scheme used to simulate solid-state dewetting problems in 2D \cite{Zhao2021}. We introduce a time-independent spatial variable $\rho$ such that $\Gamma(t)$ can be parameterized on the fixed domain $ \mathbf{I}=[0,1]$ as
$$
\Gamma(t):=\vec X(\rho,t)=(x(\rho,t),~y(\rho,t))^T:\quad\mathbf{I}\times [0, T]~\rightarrow~\mathbb{R}^2.
$$
The arc length $s$ can be expressed as $s(\rho,t)=\int_0^\rho |\partial_q\vec{X}|dq$, and we have $\partial_\rho s=|\partial_\rho\vec{X}|$.

We also introduce the functional spaces associated with the evolving curve $\Gamma(t)$ as
\begin{align*}
H^1(\mathbf{I})&=\{ f : \mathbf{I}\rightarrow \mathbb{R} \mid f\in L^2(\mathbf{I}),~ \partial_\rho f\in L^2(\mathbf{I})\},\\
H^1_0(\mathbf{I})&=\{f : \mathbf{I}\rightarrow \mathbb{R} \mid f\in H^1(\mathbf{I}),~f(0)=f(1) =0\}.
\end{align*}
where the space $L^2(\mathbf{I})$ is defined as:
$$
L^2(\mathbf{I})=\{f : \mathbf{I}\rightarrow \mathbb{R}\mid \langle f,f\rangle_{\Gamma(t)}
=\int_{\mathbf{I}}  f^2(\rho)|\partial_\rho\vec{X}| d\rho <+\infty \}.
$$
The boundary condition \eqref{eqn:weakBC1} implies that the y-component of the solution satisfies $y(\cdot,t)\in H^1_0(\mathbf{I})$.

The new variational formulation for the sharp-interface model (\ref{eqn:2diso}) with boundary conditions (\ref{eqn:weakBC1})--(\ref{eqn:weakBC3}) and initial condition \eqref{eqn:initial} is stated as follows: Given an initial open curve $\Gamma(0)=\vec X(0)\in H^1(\mathbf{I})\times H_0^1(\mathbf{I})$ with $x_c^l(0) < x_c^r(0)$, find for time $t>0$ the evolving curve $\Gamma(t)=\vec X(\cdot,t)=(x(\cdot,t), y(\cdot,t))^T \in H^1(\mathbf{I})\times H_0^1(\mathbf{I})$ and the curvature $\kappa(\cdot, t)\in H^1(\mathbf{I})$ such that
\begin{subequations}
	\label{eqn:2dweak}
	\begin{align}
	\label{eqn:2dweak1}
	&\left\langle \partial_t\vec X,~\vec n\psi\right\rangle_{\Gamma(t)} + \left\langle \partial_s\kappa,~\partial_s\psi\right\rangle_{\Gamma(t)}=0,\qquad\forall ~\psi\in H^1(\mathbf{I}),\\[0.5em]
	\label{eqn:2dweak2}
	&\left\langle\kappa\vec n,\boldsymbol{\omega}\right\rangle_{\Gamma(t)}-\left\langle\partial_s\vec X,~\partial_s\boldsymbol{\omega}\right\rangle_{\Gamma(t)} - \frac{1}{\eta}\Bigr[\frac{dx_c^l(t)}{dt}\,\omega_1(0) + \frac{dx_c^r(t)}{dt}\omega_1(1)\Bigr]\nonumber\\
	&\qquad\qquad\qquad\qquad + \;\sigma\,\Bigl[\omega_1(1) - \omega_1(0)\Bigr] = 0,\quad\forall~\boldsymbol{\omega}=(\omega_1,~\omega_2)^T\in H^1(\mathbf{I})\times H_0^1(\mathbf{I}),
	\end{align}
\end{subequations}
where $x_c^l(t)$ and $x_c^r(t)$ represent the $x$-coordinates of the left and right contact points, respectively, and are assumed to satisfy $x_c^l(t) \le x_c^r(t)$ for all $t > 0$. Equation \eqref{eqn:2dweak1} is derived using integration by parts and the zero-mass flux condition \eqref{eqn:weakBC3}, while \eqref{eqn:2dweak2} is obtained through integration by parts and the Robin-type boundary condition (the relaxed contact angle condition) \eqref{eq:Robin}. The mass/area conservation and energy dissipation properties of the proposed variational formulation \eqref{eqn:2dweak}  can also be rigorously established \cite{Zhao2021}.

For the spatial semi-discretization, let $N \geq 3$ be a positive integer. Denote the mesh size by $h = 1/N$ and the grid points by $\rho_j = jh$ for $j = 0, 1, \ldots, N$. The interval $\mathbf{I} = [0, 1]$ is then partitioned as $\mathbf{I} = \bigcup_{j=1}^{N} I_j$, where each subinterval is given by $I_j = [\rho_{j-1}, \rho_j]$. We also introduce the piecewise linear finite element spaces as
\begin{equation*}
\begin{split}
&\mathbb{K}^h:=\{\psi^{h}\in C(\mathbf{I})\mid\;\psi^{h}\mid_{I_{j}}\in \mathcal{P}_1,\quad \forall \, j=1,2,\ldots,N\}\subseteq H^1(\mathbf{I}),\\
&\mathbb{K}_0^h := \left\{ \psi^h \in \mathbb{K}^h\mid\; \psi^h(0) =0,~ \psi^h(1)=0 \right\}\subseteq  H^1_{0}(\bm{I}),
\end{split}
\end{equation*}
where $\mathcal{P}_1$ stands for the space of polynomials of degree at most 1. For the basis functions of the space $\mathbb{K}^h$, we introduce hat functions associated with the nodes $\rho_{0},\ldots,\rho_{N}$. For $j=1,\ldots,N-1$, let
\[\phi_{j}^{h}(\rho)=\left\{\begin{array}{lll}
(\rho-\rho_{j-1})/h&\rho_{j-1}\leq\rho\leq\rho_{j},\\
(\rho_{j+1}-\rho)/h&\rho_{j}\leq\rho\leq\rho_{j+1},\\
0&\text{otherwise},
\end{array}\right.\]
for $j=0$,
\[\phi_{0}^{h}(\rho)=\left\{\begin{array}{lll}
(\rho_{1}-\rho)/h&\rho_0\leq\rho\leq\rho_{1},\\
0&\text{otherwise},
\end{array}\right.\]
and for $j=N$,
\[\phi_{N}^{h}(\rho)=\left\{\begin{array}{lll}
(\rho-\rho_{N-1})/h&\rho_{N-1}\leq\rho\leq\rho_{N},\\
0&\text{otherwise},
\end{array}\right.\]
The basis functions of the space $ \mathbb{K}^h \times \mathbb{K}_0^h$ are given by $(\phi_{j}^{h},0)^{T}$ and $(0,\phi_{i}^{h})^{T}$ for
$j=0,\ldots,N$ and $i=1,\ldots,N-1$.

Let $\Gamma^h(t) = \vec{X}^h(\cdot,t) = (x^h(\cdot,t), y^h(\cdot,t))^{T} \in \mathbb{K}^h \times \mathbb{K}_0^h$ be the numerical approximation to the solution $\Gamma(t) = \vec{X}(\cdot,t)$ of the variational problem \eqref{eqn:2dweak}. The mass-lumped inner product $\langle \cdot, \cdot \rangle_{\Gamma^h(t)}^h$ on the polygonal curve $\Gamma^h(t) = \vec{X}^{h}(t)$ is defined by
$$ \left\langle u,~v\right\rangle_{\Gamma^h(t)}^h=\frac{1}{2}\sum_{j=1}^{N}\Big|\vec X^h(\rho_j,t)-\vec X^h(\rho_{j-1},t)\Big|\Big[\big(u\cdot v\big)(\rho_j^-)+\big(u\cdot v\big)(\rho_{j-1}^+)\Big],     $$
where $u(\rho_{j}^{\pm})=\lim\limits_{\rho\rightarrow \rho_j^{\pm}}u(\rho)$.

Subsequently, the semi-discrete scheme for the variational formulation \eqref{eqn:2dweak} using continuous piecewise linear elements can be stated as follows. Given an  initial curve $\Gamma^h(0) = \vec{X}^h(0) \in \mathbb{K}^h \times \mathbb{K}_0^h$ with $(x^h)_c^l(0) < (x^{h})_c^r(0)$, find for time $t > 0$ the evolving curve $\Gamma^h(t) = \vec{X}^h(\cdot, t)=(x^h(\cdot,t), y^h(\cdot,t))^T  \in \mathbb{K}^h \times \mathbb{K}_0^h$ and the curvature $\kappa^h(\cdot, t) \in \mathbb{K}^h$ such that
\begin{subequations}
	\label{eqn:2dsemi}
	\begin{align}
	\label{eqn:2dsemi1}
	&\left\langle\partial_t\bm{ X}^h,~\bm{ n}^h\psi^h\right\rangle_{\Gamma^h(t)}^h + \left\langle\partial_s\kappa^h,~\partial_s\psi^h\right\rangle_{\Gamma^h(t)}=0,\qquad\forall ~\psi^h\in \mathbb{K}^{h},\\[0.5em]
	\label{eqn:2dsemi2}
	&\left\langle\kappa^h\bm{ n}^h, \boldsymbol{\omega}^h\right\rangle_{\Gamma^h(t)}^h-\left\langle\partial_s\bm{ X}^h,~\partial_s\boldsymbol{\omega}^h\right\rangle_{\Gamma^h(t)} - \frac{1}{\eta}\Bigr[\frac{\mathrm{d}(x^h)_c^l(t)}{\mathrm{d}t}\,\omega_1^h(0) + \frac{\mathrm{d}(x^h)_c^r(t)}{\mathrm{d}t}\,\omega_1^h(1)\Bigr] \nonumber \\
	&\qquad\qquad\qquad+ \sigma\,\Bigl[\omega_1^h(1) - \omega_1^h(0)\Bigr] = 0,\quad\forall~\boldsymbol{\omega}^h=(\omega_1^h,~\omega_2^h)^T
	\in\mathbb{K}^h\times\mathbb{K}_0^h,
	\end{align}
\end{subequations}
where $(x^h)_c^l(t)$ and $(x^h)_c^r(t)$ represent the $x$-coordinates of the left and right contact points, respectively, and are assumed to satisfy $(x^h)_c^l(t) \le (x^h)_c^r(t)$ for all $t > 0$. Here, the normal vector $\bm{n}^{h}$ and the partial derivative $\partial_{s}$ are defined piecewise on each interval $I_j$ as follows:
$$\bm{n}^{h}|_{I_{j}}=-\frac{\bm{h}_{j}^{\bot}}{|\bm{h}_{j}|},\quad
\partial_{s}f|_{I_{j}}=\frac{\partial_{\rho}f}{|\partial_{\rho}\bm{X}^{h}|}\bigg|_{I_{j}}=\frac{h\partial_{\rho}f|_{I_{j}}}{|\bm{h}_{j}|},$$
where  $h_j(t)=\vec X^h(\rho_j,t)-\vec X^h(\rho_{j-1},t)$ for $j=1,\ldots,N$. Furthermore, it can be rigorously proven that the semi-discrete scheme \eqref{eqn:2dsemi} preserves mass/area and guarantees energy dissipation \cite{Zhao2021}.

For the time discretization, we choose a uniform time step size $\tau > 0$ and define the discrete time levels by $t_m = m\tau$ for integers $m \geq 0$. Let $\Gamma^m = \vec{X}^m = (x^m, y^m)^{T} \in \mathbb{K}^h \times \mathbb{K}_0^h$ denote the numerical approximation to the solution $\Gamma^{h}(t_m) = \vec{X}^h(t_m)$ of the semi-discrete scheme \eqref{eqn:2dsemi} at time $t_m$. The discrete inner product $\langle \cdot, \cdot \rangle_{\Gamma^m}^h$ and the discrete unit normal vector $\vec{n}^m$ are defined analogously to their semi-discrete counterparts.

The full discretization of the variational formulation \eqref{eqn:2dweak} is obtained by applying the backward Euler method to the semi-discrete formulation \eqref{eqn:2dsemi} in time. It can be stated as follows: Given an initial curve $\Gamma^0 = \vec{X}^{0} \in \mathbb{K}^h \times \mathbb{K}_0^h$ with $(x^0)_c^l < (x^0)_c^r$, find for each $m \geq 0$ the evolving curve $\Gamma^{m+1} = \vec{X}^{m+1}=(x^{m+1}, y^{m+1})^T \in \mathbb{K}^h \times \mathbb{K}_0^h$ and the curvature $\kappa^{m+1} \in \mathbb{K}^h$ such that
\begin{subequations}
	\label{eqn:dis2d}
	\begin{align}
	\label{eqn:dis2d1}
	&\left\langle   \frac{\vec X^{m+1}-\vec X^m}{\tau},~\vec n^m\psi^h\right\rangle_{\Gamma^m}^h + \left\langle\partial_s\kappa^{m+1},~\partial_s\psi^h\right\rangle_{\Gamma^m}=0,\quad\forall ~\psi^h\in \mathbb{K}^h,\\[0.5em]
	\label{eqn:2dsemi2}
	&\left\langle\kappa^{m+1}\bm{ n}^m, \boldsymbol{\omega}^h\right\rangle_{\Gamma^m}^h-\left\langle\partial_s\bm{ X}^{m+1},~\partial_s\boldsymbol{\omega}^h\right\rangle_{\Gamma^m} - \frac{1}{\eta\tau}\left[\Big((x^{m+1})_c^l-(x^m)_c^l\Big)\,\omega_1^h(0) + \Big((x^{m+1})_c^r-(x^m)_c^r\Big)\,\omega_1^h(1)\right] \nonumber \\
	&\qquad\qquad\qquad+ \sigma\,\Bigl[\omega_1^h(1) - \omega_1^h(0)\Bigr] = 0,\quad\forall~\boldsymbol{\omega}^h=(\omega_1^h,~\omega_2^h)^T
	\in\mathbb{K}^h\times\mathbb{K}_0^h,
	\end{align}
\end{subequations}
where $(x^m)_c^l$ and $(x^m)_c^r$ represent the $x$-coordinates of the left and right contact points, respectively, and are assumed to satisfy $(x^m)_c^l \le (x^m)_c^r$ for all $m \geq 1$. The fully discrete scheme \eqref{eqn:dis2d} is well-posed under mild conditions and is unconditionally energy stable \cite{Zhao2021}.

Moreover, for all $m \geq 0$, the area of the region enclosed by the open curve $\Gamma^m$ and the substrate is defined as:
$$A^{m}=\frac{1}{2}\sum_{j=1}^{N}(x_{j}^{m}-x_{j-1}^{m})(y_{j}^{m}+y_{j-1}^{m}).$$
Similarly, the energy of the system under the full-discretization is given by:
$$ W^{m}=\sum_{j=1}^{N}\mid\bm{h}_{j}^{m}\mid-\sigma\left((x^m)_c^r-(x^m)_c^l\right).$$


\section{High-order in time schemes}
In this section, we propose two classes of high-order temporal schemes based on a predictor-corrector strategy and backward differentiation formulae (BDF), respectively.
\subsection{Predictor-corrector time discretization}
Following the methodology in \cite{Jiang2025}, the predictor-corrector strategy is employed to discretize the semi-discrete scheme \eqref{eqn:2dsemi} at the intermediate time point $t_{m+\frac{1}{2}}$. Specifically, in \eqref{eqn:2dsemi}, we introduce the following approximation for the time derivative term:
\begin{subequations}
	\begin{align}
\partial_{t} \mathbf{X}^{h}\left(t_{m+\frac{1}{2}}\right) &\sim \frac{\mathbf{X}^{h}(t_{m+1})-\mathbf{X}^{h}(t_m)}{\tau}+\mathcal{O}\left(\tau^{2}\right),\label{eq:pct1}\\
\frac{\mathrm{d}(x^h)_c^l}{\mathrm{d} t}\left(t_{m+\frac{1}{2}}\right)&\sim \frac{(x^{h})_c^l(t_{m+1})-(x^{h})_c^l(t_{m})}{\tau}+\mathcal{O}\left(\tau^{2}\right),\label{eq:pct2}\\
\frac{\mathrm{d}(x^h)_c^r}{\mathrm{d} t}\left(t_{m+\frac{1}{2}}\right)&\sim \frac{(x^{h})_c^r(t_{m+1})-(x^{h})_c^r(t_{m})}{\tau}+\mathcal{O}\left(\tau^{2}\right).\label{eq:pct3}
\end{align}
\end{subequations}

Inspired by the above approximations and by the approximation of $\partial_s\kappa^h$ and $\partial_s\bm{X}^h$ via the midpoint rule, we derive the following scheme (denoted as the \textbf{PC-ZJB scheme}):
Given an initial curve $\Gamma^0 = \vec{X}^{0} \in \mathbb{K}^h \times \mathbb{K}_0^h$ with $(x^0)_c^l < (x^0)_c^r$, for each $m \geq 0$, find the evolving curve $\Gamma^{m+1} = \vec{X}^{m+1} =(x^{m+1}, y^{m+1})^T\in \mathbb{K}^h \times \mathbb{K}_0^h$ and the curvature $\kappa^{m+1} \in \mathbb{K}^h$ such that
\begin{subequations}
	\label{PC}
	\begin{align}
	\label{PCa}
	& \left\langle\frac{\bm{X}^{m+1}-\bm{X}^{m}}{\tau},\widetilde{\bm{n}}^{m+1/2}\psi^{h}\right\rangle_{\widetilde{\Gamma}^{m+1/2}}^{h}+\left\langle\partial_{s}\left(\frac{\kappa^{m+1}+\kappa^{m}}{2}\right),\partial_{s}\psi^{h}\right\rangle_{\widetilde{\Gamma}^{m+1/2}}=0\quad\forall~\psi^h\in \mathbb{K}^h,\\[0.5em]
	\label{PCb}
& \left\langle\frac{\kappa^{m+1}+\kappa^{m}}{2}\widetilde{\bm{n}}^{m+1/2},\boldsymbol{\omega}^{h}\right\rangle_{\widetilde{\Gamma}^{m+1/2}}^{h}-\left\langle\partial_{s}\left(\frac{\bm{X}^{m+1}+\bm{X}^{m}}{2}\right),\partial_{s}\boldsymbol{\omega}^{h}\right\rangle_{\widetilde{\Gamma}^{m+1/2}}
 + \sigma\,\Bigl[\omega_1^h(1) - \omega_1^h(0)\Bigr] \nonumber \\
	&-\frac{1}{\eta\tau}\Bigr[\Big((x^{m+1})_c^l-(x^{m})_c^l\Big)\,\omega_1^h(0) +\Big ((x^{m+1})_c^r-(x^{m})_c^r\Big)\,\omega_1^h(1)\Bigr] =0\quad\forall~\boldsymbol{\omega}^h=(\omega_1^h,~\omega_2^h)^T
	\in\mathbb{K}^h\times\mathbb{K}_0^h,
	\end{align}
\end{subequations}
where $\widetilde{\Gamma}^{m+1/2} = \widetilde{\bm{X}}^{m+1/2}(\bm{I})$, with $\widetilde{\bm{X}}^{m+1/2}$ being the solution of the ZJB scheme \eqref{eqn:dis2d} using a half time step $\tau/2$. Here, both the normal vector $\widetilde{\bm{n}}^{m+1/2}$ and the partial derivative $\partial_{s}$ are defined piecewise over the curve $\widetilde{\Gamma}^{m+1/2}$ as follows:
 $$\widetilde{n}^{m+1/2}\mid_{I_{j}}=-\frac{(\partial_{\rho}\widetilde{\bm{X}}^{m+1/2})^{\bot}}{|\partial_{\rho}\widetilde{\bm{X}}^{m+1/2}|}\bigg|_{I_{j}}=-\frac{(\widetilde{\bm{h}}^{m+1/2}_{j})^{\bot}}{|\widetilde{\bm{h}}^{m+1/2}_{j}|}=\widetilde{\bm{n}}_j^{m+1/2},\quad
\partial_{s}f\mid_{I_{j}}=\frac{\partial_{\rho}f}{|\partial_{\rho}\widetilde{\bm{X}}^{m+1/2}|}\bigg|_{I_{j}}=\frac{h\partial_{\rho}f|_{I_{j}}}{|\widetilde{\bm{h}}^{m+1/2}_{j}|},$$
for $j=1,\ldots,N$.

\begin{algorithm}(\textbf{PC-ZJB algorithm for solid-state dewetting})\label{algorithm-pc}\\
\textbf{Step~0}. Given an initial curve $\Gamma^{0} = \vec{X}^{0}\in \mathbb{K}^h \times \mathbb{K}_0^h$, a number of grid points $N$, and a time step size $\tau$, we construct the polygon $\Gamma^{0}$ using an equal partition of the polar angle.\\
\textbf{Step~1}. Using $\bm{X}^{0}$ as the initial input, we compute $\kappa^{0}$ following the procedure outlined in Remark 2.1 of \cite{Jiang2024}. The computation entails applying the least squares method to the following weak formulation:
$$\left\langle\kappa^{0}\bm{ n}^{0}, \boldsymbol{\omega}^h\right\rangle_{\Gamma^0}^h-\left\langle\partial_s\bm{ X}^0,~\partial_s\boldsymbol{\omega}^h\right\rangle_{\Gamma^0}=0\quad\forall~\boldsymbol{\omega}^h
	\in\mathbb{K}^h\times\mathbb{K}_0^h.$$ 
  Set $m=0$.\\
\textbf{Step~2}.
\begin{itemize}
\item \textbf{Predictor}: Compute $\widetilde{\bm{X}}^{m+1/2}$ by using the ZJB scheme \eqref{eqn:dis2d} for one time step with time step size $\tau/2$.
\item \textbf{Corrector}: Compute $\bm{X}^{m+1},\kappa^{m+1}$ by using the PC-ZJB scheme \eqref{PC} for one time step with  time step size $\tau$.
\end{itemize}
\textbf{Step~3}. Update $m=m+1$. If $m<T/\tau$, then go back to \textbf{Step~2}; otherwise, stop the algorithm and output the data.
\end{algorithm}

Similar to the proof of Theorem 4.1 in \cite{Zhao2021}, the well-posedness of the PC-ZJB scheme \eqref{PC} can be established under mild assumptions on the predicted curve $\widetilde{\Gamma}^{m+1/2}$. Furthermore, we define the space $\mathbb{X}^{h} = \mathbb{K}^h \times \mathbb{K}_0^h$. The normal vector $\widetilde{\bm{n}}_j^{m+1/2} = (\widetilde{n}^{m+1/2}_{j,1}, \widetilde{n}^{m+1/2}_{j,2})^{T}$ of $\widetilde{\bm{X}}^{m+1/2}$ on the interval $I_j$ is defined for each $j = 1, \ldots, N$.

\begin{thm}(Well-posedness)\label{thm:PC}
Assume that the following conditions are satisfied:
\begin{equation}\label{equ:well-posedness}
\begin{aligned}
&(i)\quad(\widetilde{n}^{m+1/2}_{1,1})^{2}+(\widetilde{n}^{m+1/2}_{N,1})^{2}>0;\\[0.5em]
&(ii)\quad \min_{1\leq j \leq N}|\widetilde{\bm{h}}^{m+1/2}_{j}|=\min_{1\leq j \leq N}|\widetilde{\bm{X}}^{m+1/2}(\rho_{j})-\widetilde{\bm{X}}^{m+1/2}(\rho_{j-1})|>0.
\end{aligned}
\end{equation}
Then, the $N$-dimensional linear system \eqref{PC} is well-posed, i.e., there exists a unique solution $(\bm{X}^{m+1}, \kappa^{m+1}) \in \mathbb{X}^{h} \times \mathbb{K}^h$.
\end{thm}
\begin{proof}
To establish well-posedness, it suffices to demonstrate that the following homogeneous linear system for $(\bm{X},\kappa)\in\mathbb{X}^{h}\times\mathbb{K}^h$ admits only the trivial solution:
\begin{subequations}\label{PC-well}
\begin{align}
 \left\langle\bm{X},\widetilde{\bm{n}}^{m+1/2}\psi^{h}\right\rangle_{\widetilde{\Gamma}^{m+1/2}}^{h}&+\tau/2\left\langle\partial_{s}\kappa,\partial_{s}\psi^{h}\right\rangle_{\widetilde{\Gamma}^{m+1/2}}=0\quad\forall ~\psi^h\in \mathbb{K}^h,\label{PC-wella}\\
 \left\langle\kappa\widetilde{\bm{n}}^{m+1/2},\boldsymbol{\omega}^{h}\right\rangle_{\widetilde{\Gamma}^{m+1/2}}^{h}&-\left\langle\partial_{s}\bm{X},\partial_{s}\boldsymbol{\omega}^{h}\right\rangle_{\widetilde{\Gamma}^{m+1/2}}
-\frac{2}{\eta\tau}\Bigr[x_c^l\,\omega_1^h(0) + x_c^r\,\omega_1^h(1)\Bigr] =0\quad\forall~\boldsymbol{\omega}^h
	\in\mathbb{X}^{h},\label{PC-wellb}
\end{align}
\end{subequations}
where $\bm{X} = (x, y)^T$ satisfies $y(\rho=0) = y(\rho=1) = 0$, and
\begin{equation}\label{eq:20}
x_{c}^l = x(\rho=0)=x(0), \quad x_{c}^r = x(\rho=1)=x(1).
\end{equation}
Choosing the test functions $\psi^{h} = \kappa \in \mathbb{K}^{h}$ in \eqref{PC-wella} and $\boldsymbol{\omega}^{h} = \bm{X} \in \mathbb{X}^{h}$ in \eqref{PC-wellb}, we obtain
\begin{equation}\label{eq:1}
\tau/2\left\langle\partial_{s}\kappa,\partial_{s}\kappa\right\rangle_{\widetilde{\Gamma}^{m+1/2}} + \left\langle\partial_{s}\bm{X},\partial_{s}\bm{X}\right\rangle_{\widetilde{\Gamma}^{m+1/2}} + \frac{2}{\eta\tau}\left[x_c^l x(0) + x_c^r x(1)\right] = 0.
\end{equation}
Since $\bm{X} \in \mathbb{X}^{h}$, $\kappa \in \mathbb{K}^{h}$, and from notation \eqref{eq:20}, it follows that
\begin{equation}\label{eq:2}
x_c^l = x_c^r = 0, \quad \bm{X}(\rho) \equiv \bm{X}^{c} \in \mathbb{R}^2, \quad \kappa(\rho) \equiv \kappa^{c} \in \mathbb{R} \quad \forall \rho \in [0,1].
\end{equation}
It is obvious that $\bm{X}^{c} =\bm{X}(\rho=0)=\bm{0}$, since $x(\rho=0)=x_c^l =0$ and $y(\rho=0)=0$.

We now prove that $\kappa^{c} = 0$. Substituting \eqref{eq:2} into \eqref{PC-wellb} yields
\begin{equation}\label{eq:3}
0 = \left\langle\kappa^{c}\widetilde{\bm{n}}^{m+1/2},\boldsymbol{\omega}^{h}\right\rangle_{\widetilde{\Gamma}^{m+1/2}}^{h} = \kappa^{c} \left\langle\widetilde{\bm{n}}^{m+1/2},\boldsymbol{\omega}^{h}\right\rangle_{\widetilde{\Gamma}^{m+1/2}}^{h}, \quad \forall ~\boldsymbol{\omega}^{h} \in \mathbb{X}^{h}.
\end{equation}
Now, choose the test function $\boldsymbol{\omega}^{h} \in \mathbb{X}^{h}$ such that
\begin{equation}\label{eq:4}
\boldsymbol{\omega}^{h}(\rho=0) = (\widetilde{n}^{m+1/2}_{1,1}, 0)^{\top}, \quad \boldsymbol{\omega}^{h}(\rho=1) = (\widetilde{n}^{m+1/2}_{N,1}, 0)^{\top}, \quad \boldsymbol{\omega}^{h}(\rho=\rho_j) = \bm{0} \quad \text{for } 1 \leq j \leq N-1.
\end{equation}
Then, from \eqref{equ:well-posedness} and the definition of the mass-lumped inner product, we have
\begin{align}
\left\langle\widetilde{\bm{n}}^{m+1/2},\boldsymbol{\omega}^{h}\right\rangle_{\widetilde{\Gamma}^{m+1/2}}^{h}
&= \frac{1}{2} \sum_{j=1}^{N} |\widetilde{\bm{h}}^{m+1/2}_{j}| \left[ \boldsymbol{\omega}^{h}(\rho_{j-1}) \cdot \widetilde{\bm{n}}^{m+1/2}_{j} + \boldsymbol{\omega}^{h}(\rho{j}) \cdot \widetilde{\bm{n}}^{m+1/2}_{j} \right] \notag \\
&= \frac{1}{2} |\widetilde{\bm{h}}^{m+1/2}_{1}| \boldsymbol{\omega}^{h}(0) \cdot \widetilde{\bm{n}}^{m+1/2}_{1} + \frac{1}{2} |\widetilde{\bm{h}}^{m+1/2}_{N}| \boldsymbol{\omega}^{h}(1) \cdot \widetilde{\bm{n}}^{m+1/2}_{N} \notag \\
&= \frac{1}{2} |\widetilde{\bm{h}}^{m+1/2}_{1}| (\widetilde{n}^{m+1/2}_{1,1})^2 + \frac{1}{2} |\widetilde{\bm{h}}^{m+1/2}_{N}| (\widetilde{n}^{m+1/2}_{N,1})^2 \notag \\
&\geq C \left[ (\widetilde{n}^{m+1/2}_{1,1})^2 + (\widetilde{n}^{m+1/2}_{N,1})^2 \right] > 0, \label{eq:5}
\end{align}
where $C = \frac{1}{2} \min\left\{ |\widetilde{\bm{h}}^{m+1/2}_{1}|,|\widetilde{\bm{h}}^{m+1/2}_{N}|\right\}$. Combining \eqref{eq:3} and \eqref{eq:5}, we conclude that $\kappa^{c} = 0$, and hence $\kappa(\rho)\equiv 0$ for all $\rho \in [0,1]$.
This completes the proof of well-posedness.

\end{proof}

 Define the mesh ratio indicator $\Psi^{m}$ for the curve $\Gamma^{m}$ as
\begin{equation}\label{116}
\Psi^{m}=\frac{\max_{1\leq i\leq N}| \bm{h}_{i}^{m}|}{\min_{1\leq i\leq N}| \bm{h}_{i}^{m}|}, ~\text{with}~| \bm{h}_{i}^{m}|=|\vec X^m(\rho_i)-\vec X^m(\rho_{i-1})|\quad m\geq0.
\end{equation}
Similar to previous work \cite{Bao2021, Zhao2021, Gan2025}, we have the following proposition regarding the long-time behavior and equilibrium state of the solution to \eqref{PC}.

\begin{prop}\label{prop1}
Let $(\bm{X}^{m},\kappa^{m})\in\mathbb{X}^{h}\times\mathbb{K}^{h}$ be a solution to \eqref{PC} with $\min_{1\leq j\leq N}| \bm{h}_{j}^{m} |>0$ for all $m\geq0$. Suppose that as $m\rightarrow \infty$, $\bm{X}^{m}$ and $\kappa^{m}$ converge to an equilibrium state $\Gamma^{e}=\bm{X}^{e}=(x^{e},y^{e})^{\top}\in \mathbb{X}^{h}$ and $\kappa^{e}\in \mathbb{K}^{h}$, respectively, and that the predictor $\widetilde{\bm{X}}^{m+1/2}$ also converges to $\Gamma^e$. Assume further that $\min_{1\leq i\leq N}| \bm{h}_{i}^{e}|>0$, where $\bm{h}_{i}^{e}=\bm{X}^{e}(\rho_{i})-\bm{X}^{e}(\rho_{i-1})$ for $1\leq i\leq N$. Then the following hold:
\begin{align}
\kappa^{e}(\rho)&\equiv \kappa^{c},~\forall\rho\in[0,1],~\text{with}~\kappa^{c}\neq 0~\text{a constant}, \label{23a}\\
\lim_{m\rightarrow+\infty}\Psi^{m}&=\Psi^{e}=\frac{\max_{1\leq i\leq N}| \bm{h}_{i}^{e}|}{\min_{1\leq i\leq N}| \bm{h}_{i}^{e}|}=1,\label{23b}
\end{align}
Moreover, let $\theta_{e}^{l}$ and $\theta_{e}^{r}$ be the left and right contact angles of $\Gamma^e$, respectively. Then there exist constants $h_{0} > 0$ and $C_0 > 0$ such that for all $0 < h < h_{0}$,
\begin{equation}\label{23c}
|\cos(\theta_{e}^{l})-\sigma|\leq C_{0}h \quad \text{and} \quad |\cos(\theta_{e}^{r})-\sigma|\leq C_{0}h.
\end{equation}
\end{prop}
\begin{proof}
Under the assumptions of this proposition, passing to the limit as $m\rightarrow+\infty$ in \eqref{PC}, we obtain that the equilibrium solution $(\bm{X}^{e},\kappa^{e})\in \mathbb{X}^{h}\times\mathbb{K}^{h}$ satisfies the following variational problem:
\begin{subequations}
\begin{align}
\left\langle\partial_{s}\kappa^{e},\partial_{s}\psi^{h}\right\rangle_{\Gamma^{e}}
 &=0\quad\forall~\psi^{h}\in \mathbb{K}^{h},\label{long-timea}\\
 \left\langle\kappa^{e}\bm{n}^{e},\boldsymbol{\omega}^{h}\right\rangle_{\Gamma^{e}}^{h}
 -\left\langle\partial_{s}\bm{X}^{e},\partial_{s}\boldsymbol{\omega}^{h}\right\rangle_{\Gamma^{e}}+\sigma\,\Bigl[\omega_1^h(1) - \omega_1^h(0)\Bigr]
 &=0\quad\forall~\boldsymbol{\omega}^{h}=(\omega_1^h,~\omega_2^h)^T\in \mathbb{X}^{h}.\label{long-timeb}
\end{align}
\label{long-time}
\end{subequations}
The equilibrium equation \eqref{long-time} agrees with its counterpart in the ZJB scheme \eqref{eqn:dis2d}. The proof can be found in Proposition 3.3 of \cite{Zhao2021}.
\end{proof}
\subsection{Backward differentiation formulae time discretization}
Following the methodology in \cite{Jiang-BDF2024}, the backward differentiation formula (BDF) is employed to discretize the semi-discrete scheme \eqref{eqn:2dsemi} at the time level $t_{m+1}$. Taking BDF2 as an example, the time derivative terms in \eqref{eqn:2dsemi} can be approximated as follows:
\begin{subequations}
	\begin{align}
		\partial_{t} \mathbf{X}^{h}\left(t_{m+1}\right) &\sim \frac{\frac{3}{2}\mathbf{X}^{h}(t_{m+1})-2\mathbf{X}^{h}(t_m)+\frac{1}{2}\mathbf{X}^{h}(t_{m-1})}{\tau}+\mathcal{O}\left(\tau^{2}\right),\label{eq:bdf2-1}\\
		\frac{\mathrm{d} (x^h)_c^l}{\mathrm{d} t}\left(t_{m+1}\right)&\sim \frac{\frac{3}{2}(x^{h})_c^l(t_{m+1})-2(x^{h})_c^l(t_{m})+\frac{1}{2}(x^{h})_c^l(t_{m-1})}{\tau}+\mathcal{O}\left(\tau^{2}\right),\label{eq:bdf2-2}\\
		\frac{\mathrm{d} (x^h)_c^r}{\mathrm{d} t}\left(t_{m+1}\right)&\sim \frac{\frac{3}{2}(x^{h})_c^r(t_{m+1})-2(x^{h})_c^r(t_{m})+\frac{1}{2}(x^{h})_c^r(t_{m-1})}{\tau}+\mathcal{O}\left(\tau^{2}\right).\label{eq:bdf2-3}
	\end{align}
\end{subequations}

Inspired by the above approximations, we derive the following scheme (denoted as the \textbf {BDF$k$-ZJB scheme}):
Given an initial curves $\Gamma^p = \vec{X}^{p} \in \mathbb{K}^h \times \mathbb{K}_0^h$ with $(x^p)_c^l < (x^p)_c^r$ for $0 \leq p \leq k-1$, for each $m \geq k-1$, find the evolving curve $\Gamma^{m+1} = \vec{X}^{m+1}=(x^{m+1},y^{m+1})^{T} \in \mathbb{K}^h \times \mathbb{K}_0^h$ and the curvature $\kappa^{m+1} \in \mathbb{K}^h$ such that
\begin{subequations}
	\label{BDFk}
	\begin{align}
	\label{BDFka}
	& \left\langle\frac{a\bm{X}^{m+1}-\bm{\widehat{X}}^{m}}{\tau},\bm{\tilde{n}}^{m+1}\psi^{h}\right\rangle_{\widetilde{\Gamma}^{m+1}}^{h}+\left\langle\partial_{s}\kappa^{m+1},\partial_{s}\psi^{h}\right\rangle_{\widetilde{\Gamma}^{m+1}}
 =0 \quad\forall~ \psi^h\in \mathbb{K}^h,\\[0.5em]
\label{BDFkb}
& \left\langle\kappa^{m+1}\bm{\tilde{n}}^{m+1},\boldsymbol{\omega}^{h}\right\rangle_{\widetilde{\Gamma}^{m+1}}^{h}-\left\langle
 \partial_{s}\bm{X}^{m+1},\partial_{s}\boldsymbol{\omega}^{h}\right\rangle_{\widetilde{\Gamma}^{m+1}}+\sigma\,\Bigl[\omega_1^h(1) - \omega_1^h(0)\Bigr] \nonumber \\
	&-\frac{1}{\eta\tau}\Bigr[(a(x^{m+1})_c^l-(\widehat{x}^{m})_c^l)\,\omega_1^h(0) + (a(x^{m+1})_c^l-(\widehat{x}^{m})_c^l)\,\omega_1^h(1)\Bigr]=0 \quad\forall~\boldsymbol{\omega}^h
	\in\mathbb{K}^h\times\mathbb{K}_0^h,
	\end{align}
\end{subequations}
where the coefficients $a$ and $\bm{\widehat{X}}^{m}$ for different BDF orders are defined as:
\begin{subequations}
 \begin{align}
 &\text{BDF2}: a=\frac{3}{2},~\bm{\widehat{X}}^{m}=2\bm{X}^{m}-\frac{1}{2}\bm{X}^{m-1},\label{117a}\\
 &\text{BDF3}: a=\frac{11}{6},~\bm{\widehat{X}}^{m}=3\bm{X}^{m}-\frac{3}{2}\bm{X}^{m-1}+\frac{1}{3}\bm{X}^{m-2},\label{117b}\\
 &\text{BDF4}: a=\frac{25}{12},~\bm{\widehat{X}}^{m}=4\bm{X}^{m}-3\bm{X}^{m-1}+\frac{4}{3}\bm{X}^{m-2}-\frac{1}{4}\bm{X}^{m-3}.\label{117c}
 \end{align}
 \label{117}
 \end{subequations}
Here, $(\widehat{x}^{m})_c^l$ and $(\widehat{x}^{m})_c^r$ are obtained by replacing the vector components $\bm{X}^{m-p}$ in $\hat{\bm{X}}^{m}$ with $(x^{m-p})_c^l$ and $(x^{m-p})_c^r$ for $p = 0, \ldots, k-1$, respectively. The predicted curve $\widetilde{\Gamma}^{m+1}$, parameterized by $\widetilde{\bm{X}}^{m+1}$, is generated by the BDF$(k-1)$-ZJB scheme. Note that its first-order case (BDF1-ZJB) corresponds to the original ZJB scheme in \eqref{eqn:dis2d}. The normal vector $\widetilde{\bm{n}}^{m+1}$ and the partial derivative $\partial_{s}$ are both defined piecewise on $\widetilde{\Gamma}^{m+1}$ as follows:
 $$\widetilde{n}^{m+1}\mid_{I_{j}}=-\frac{(\partial_{\rho}\widetilde{\bm{X}}^{m+1})^{\bot}}{|\partial_{\rho}\widetilde{\bm{X}}^{m+1}|}\bigg|_{I_{j}}=-\frac{(\widetilde{\bm{h}}^{m+1}_{j})^{\bot}}{|\widetilde{\bm{h}}^{m+1}_{j}|}=\widetilde{\bm{n}}_j^{m+1},\quad
\partial_{s}f\mid_{I_{j}}=\frac{\partial_{\rho}f}{|\partial_{\rho}\widetilde{\bm{X}}^{m+1}|}\bigg|_{I_{j}}=\frac{h\partial_{\rho}f|_{I_{j}}}{|\widetilde{\bm{h}}^{m+1}_{j}|},$$
for $j=1,\ldots,N$. The initial curves $\Gamma^{p}$ for $1 \leq p \leq k-1$ can be obtained, with implementation details provided in \cite{Jiang-BDF2024}.

Similar to the proof of Theorem \ref{thm:PC}, the well-posedness of the BDFk-ZJB scheme \eqref{BDFk} can be established under mild assumptions on the predicted curve $\widetilde{\Gamma}^{m+1}$.
The normal vector $\widetilde{\bm{n}}_j^{m+1} = (\widetilde{n}^{m+1}_{j,1}, \widetilde{n}^{m+1}_{j,2})^{T}$ of $\widetilde{\bm{X}}^{m+1}$ on the interval $I_j$ is defined for each $j = 1, \ldots, N$.

\begin{thm}(Well-posedness)\label{thm:BDF}
Assume that the following conditions are satisfied:
\begin{equation}\label{equ:well-posedness-BDF}
\begin{aligned}
&(i) \quad (\widetilde{n}^{m+1}_{1,1})^{2} + (\widetilde{n}^{m+1}_{N,1})^{2} > 0; \\[0.5em]
&(ii) \quad \min_{1\leq j \leq N} |\widetilde{\bm{h}}^{m+1}_{j}| = \min_{1\leq j \leq N} |\widetilde{\bm{X}}^{m+1}(\rho_{j}) - \widetilde{\bm{X}}^{m+1}(\rho_{j-1})| > 0.
\end{aligned}
\end{equation}
Then the $N$-dimensional linear system \eqref{BDFk} is well-posed, i.e., there exists a unique solution $(\bm{X}^{m+1}, \kappa^{m+1}) \in \mathbb{X}^{h} \times \mathbb{K}^h$.
\end{thm}
\begin{proof}
To establish well-posedness, it suffices to demonstrate that the following homogeneous linear system
for $(\bm{X},\kappa)\in  \mathbb{X}^{h}\times \mathbb{K}^h$
has only the trivial solution:
\begin{subequations}\label{BDF-well}
 \begin{align}
 \left\langle a\bm{X},\widetilde{\bm{n}}^{m+1}\psi^{h}\right\rangle_{\widetilde{\Gamma}^{m+1}}^{h}&+\tau\left\langle\partial_{s}\kappa,\partial_{s}\psi^{h}\right\rangle_{\widetilde{\Gamma}^{m+1}}=0\quad
 \forall~\psi^{h}\in  \mathbb{K}^h,\label{BDF-wella}\\
 \left\langle\kappa\widetilde{\bm{n}}^{m+1},\boldsymbol{\omega}^{h}\right\rangle_{\widetilde{\Gamma}^{m+1}}^{h}&-\left\langle\partial_{s}\bm{X},\partial_{s}\boldsymbol{\omega}^{h}\right\rangle_{\widetilde{\Gamma}^{m+1}}
-\frac{a}{\eta\tau}\Bigr[x_c^l\,\omega_1^h(0) + x_c^r\,\omega_1^h(1)\Bigr] =0\quad \forall~\boldsymbol{\omega}^{h}\in \mathbb{X}^{h},\label{BDF-wellb}
 \end{align}
 \end{subequations}
where $\bm{X} = (x, y)^T$ satisfies $y(0) = y(1) = 0$, with $x_{c}^l = x(0)$ and $x_{c}^r = x(1)$.
The subsequent proof is analogous to that of Theorem \ref{thm:PC}.
\end{proof}

\begin{prop}\label{prop2}
Let $(\bm{X}^{m},\kappa^{m})\in\mathbb{X}^{h}\times\mathbb{K}^{h}$ be a solution to \eqref{BDFk} with $\min_{1\leq j\leq N}| \bm{h}_{j}^{m} |>0$ for all $m\geq0$. Suppose that as $m\rightarrow \infty$, $\bm{X}^{m}$ and $\kappa^{m}$ converge to an equilibrium state $\Gamma^{e}=\bm{X}^{e}=(x^{e},y^{e})^{\top}\in \mathbb{X}^{h}$ and $\kappa^{e}\in \mathbb{K}^{h}$, respectively, and that the predictor $\widetilde{\bm{X}}^{m+1}$ also converges to $\Gamma^e$. Assume further that $\min_{1\leq i\leq N}| \bm{h}_{i}^{e}|>0$, where $\bm{h}_{i}^{e}=\bm{X}^{e}(\rho_{i})-\bm{X}^{e}(\rho_{i-1})$ for $1\leq i\leq N$. Then the following hold:
\begin{align}
\kappa^{e}(\rho)&\equiv \kappa^{c},~\forall\rho\in[0,1],~\text{with}~\kappa^{c}\neq 0~\text{a constant}, \label{24a}\\
\lim_{m\rightarrow+\infty}\Psi^{m}&=\Psi^{e}=\frac{\max_{1\leq i\leq N}| \bm{h}_{i}^{e}|}{\min_{1\leq i\leq N}| \bm{h}_{i}^{e}|}=1,\label{24b}
\end{align}
Moreover, let $\theta_{e}^{l}$ and $\theta_{e}^{r}$ be the left and right contact angles of $\Gamma^e$, respectively. Then there exist constants $h_{0} > 0$ and $C_0 > 0$ such that for all $0 < h < h_{0}$,
\begin{equation}\label{24c}
|\cos(\theta_{e}^{l})-\sigma|\leq C_{0}h \quad \text{and} \quad |\cos(\theta_{e}^{r})-\sigma|\leq C_{0}h.
\end{equation}
\end{prop}
\begin{proof}
Under the assumptions of this proposition, passing to the limit as $m \rightarrow +\infty$ in \eqref{BDFk} shows that the variational problem reduces to \eqref{long-time}.
\end{proof}

\section{Numerical results}
In this section, we first conduct convergence rate tests for our proposed high-order temporal schemes. Subsequently, we report the temporal evolution of the total free energy, area, contact angles, and mesh ratio indicator. Finally, we apply the proposed schemes to simulate the evolution of complex curves.

All simulations are performed with a semi-elliptical initial shape, defined by a semi-major axis $a = 2$ and a semi-minor axis $b = 1$, unless otherwise stated. The contact line mobility is fixed at 
$\eta = 100$, and $\sigma=\cos(\theta_i)$, where $\theta_i$ is the Young's angle.

\subsection{Cauchy test for the convergence order}
Following the approach in \cite{Zhao2021,Jiang2024}, the error between two curves $\Gamma_{1}$ and $\Gamma_{2}$ is measured by their manifold distance. Specifically, let $\Omega_{1}$ and $\Omega_{2}$ be the regions bounded by $\Gamma_{1}$ and $\Gamma_{2}$, respectively. The manifold distance between $\Gamma_{1}$ and $\Gamma_{2}$ is defined as
$$\text{M}(\Gamma_{1},\Gamma_{2})=|(\Omega_{1}\setminus\Omega_{2})\cup(\Omega_{2}\setminus\Omega_{1})|
=2|\Omega_{1}\cup\Omega_{2}|-|\Omega_{1}|-|\Omega_{2}|,$$
where $|\Omega|$ denotes the area of $\Omega$.

We evaluate the PC-ZJB \eqref{PC} and BDF$k$-ZJB \eqref{BDFk} schemes using the Cauchy-type convergence test described in \cite{Garcke-Jiang2024}. The error and convergence order are defined as follows:
$$E_{M}(T,\tau_{1},\tau_{2})=\text{M}(\Gamma_{h_{1}}^{T/\tau_{1}},\Gamma_{h_{2}}^{T/\tau_{2}}),
~~\text{Order}=\log\left(\frac{E_{M}(T,\tau_{1},\tau_{2})}{E_{M}(T,\tau_{2},\tau_{3})}\right)\Big/\log\left(\frac{\tau_{1}}{\tau_{2}}\right),$$
where $\tau_{1},h_{1}$ and $\tau_{2},h_{2}$ satisfy the given path.

\begin{figure}[h!]
\hspace{-7mm}
\centering
\includegraphics[width=5.8in,height=3.5in]{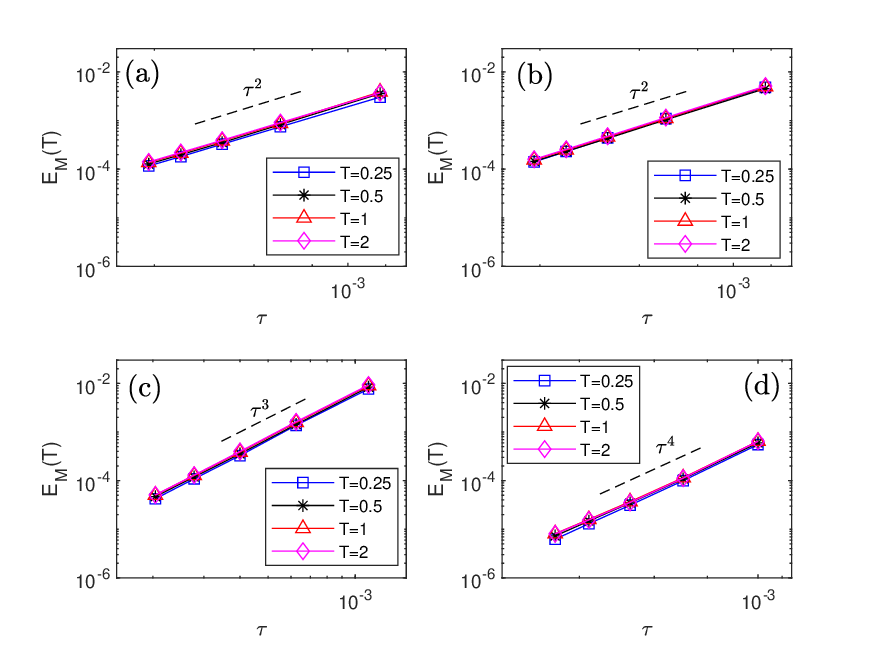}
\vspace{-5mm}
\caption{Log-log plots of the numerical errors $E_{M}$ at four different times (i.e., $T = 0.25, 0.5, 1, 2$) are presented for: (a) the PC-ZJB scheme \eqref{PC}, where the Cauchy path chosen as 
 $\tau=0.05h$; (b) the  BDF2-ZJB scheme\eqref{BDFk}, with $\tau=0.05h$; (c) the  BDF3-ZJB scheme\eqref{BDFk} , with $\tau=0.01h^{\frac{2}{3}}$; and (d) the  BDF4-ZJB scheme\eqref{BDFk},  with  $\tau=0.01h^{\frac{1}{2}}$. The Young's angle is chosen as $\theta_i=5\pi/6$.}
\label{fig:loglog1}
\end{figure}

Figure \ref{fig:loglog1} showcases the numerical errors $E_{M}$ versus the time step $\tau$ for the PC-ZJB scheme \eqref{PC} and the BDF$k$-ZJB scheme \eqref{BDFk} at four different times. As anticipated, robust convergence in terms of the manifold distance is observed: quadratic for the PC-ZJB  scheme and $k$-th order for the BDF$k$-ZJB scheme.

We now present a convergence analysis, using a Cauchy-type test, that compares the numerical equilibrium shapes from the PC-ZJB \eqref{PC} and BDF$k$-ZJB \eqref{BDFk} schemes against the theoretical Wulff shape \cite{wulff1901}. For isotropic surface energy in solid-state dewetting, the Wulff shape is a circular arc intersecting the substrate at the Young's angle $\theta_i$ \cite{Bao2017,Peng1998}. Its radius is fixed by area conservation, and the shape has the analytical expression:
\begin{equation}\label{eq:Wulff}
\begin{cases}
x(u) = -r\sin(\theta_i(1-2u)), \\
y(u) = -r\cos(\theta_i)+r\cos(\theta_i(1-2u)),
\end{cases}
\quad u \in [0, 1],
\end{equation}
where $r=\sqrt{\frac{|\Omega_0|}{\theta_i-\sin(\theta_i)\cos(\theta_i)}}$. Here, $|\Omega_0|$ denotes the area enclosed by the initial curve $\Gamma(0)$ and the substrate, and $\theta_i$ is the Young's angle.

Denoting the Wulff shape \eqref{eq:Wulff} as $\Gamma^{\text{Wulff}}$ and the numerical equilibrium shape as $\Gamma^e$, the error and convergence order are defined as
$$
E_M^{\text{Wulff}}(\tau, h) = \text{M}\left( \Gamma^{\text{Wulff}}, \Gamma^e(\tau, h) \right), \quad \text{Order} = \log\left( \frac{E_M^{\text{Wulff}}(\tau_1, h_1)}{E_M^{\text{Wulff}}(\tau_2, h_2)} \right)\Big/\log\left( \frac{\tau_1}{\tau_2} \right).
$$
The equilibrium state is considered to be reached when $\frac{W^{m}-W^{m+1}}{\tau}\leq\epsilon$ \cite{Zhao2021}, where $\epsilon$ is a small tolerance parameter set to $\epsilon=10^{-12}$ in our simulations.

\begin{figure}[h!]
\hspace{-7mm}
\centering
\includegraphics[width=4.5in,height=1.8in]{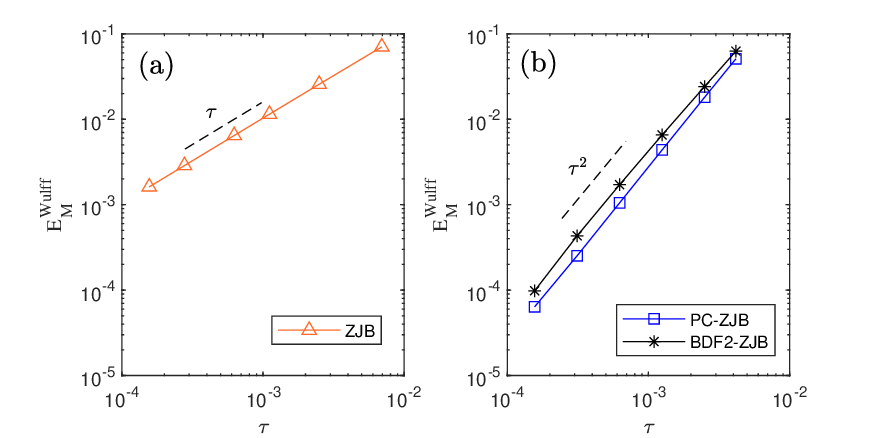}
\vspace{-3mm}
\caption{Log-log plots of the numerical errors $E_{M}^{\mathrm{Wulff}}$ are presented for: (a) the ZJB scheme \eqref{eqn:dis2d}, where the Cauchy path is chosen as $\tau=h^2$; (b) the PC-ZJB scheme \eqref{PC} and the BDF2-ZJB scheme \eqref{BDFk}, with $\tau=0.05h$. The Young's angle is chosen as $\theta_i=5\pi/6$.}
\label{fig:loglog-wulff}
\end{figure}

\begin{table}[h!]
    \centering
    \caption{A convergence order test is conducted for the PC-ZJB scheme \eqref{PC} and the BDF$k$-ZJB schemes ($k=2,3,4$) \eqref{BDFk} in the equilibrium state. The Cauchy  path is chosen as $\tau = 0.05h$ for the PC-ZJB and BDF2-ZJB schemes, $\tau = 0.025h^{2/3}$ for the BDF3-ZJB scheme, and $\tau = 0.0125h^{1/2}$ for the BDF4-ZJB scheme.}
\begin{adjustbox}{max width=\textwidth}
    \begin{tabular}{cccccccccccc}
        \toprule
        \multicolumn{3}{c}{\textbf{PC-ZJB}}&\multicolumn{3}{c}{\textbf{BDF2-ZJB}} & \multicolumn{3}{c}{\textbf{BDF3-ZJB}}& \multicolumn{3}{c}{\textbf{BDF4-ZJB}} \\
        \cmidrule(lr){1-3} \cmidrule(lr){4-6} \cmidrule(lr){7-9}  \cmidrule(lr){10-12}  
          $\tau$ & $E_M^{\mathrm{Wulff}}$  & Order &  $\tau$ & $E_M^{\mathrm{Wulff}}$  & Order & $\tau$ & $E_M^{\mathrm{Wulff}}$ & Order   & $\tau$ & $E_M^{\mathrm{Wulff}}$ & Order\\
        \midrule
        1/400  & 1.81E-02 & ---     &1/400  & 2.42E-02 & ---                & 1/360  & 1.04E-02 & ---            &1/320    &2.93E-02 &--- \\
        1/800 & 4.37E-03 & 2.0484    &1/800 & 6.55E-03 & 1.8842             & 1/640  & 1.83E-03 & 3.0209   &1/640  &1.63E-03 &4.1694\\
        1/1600 & 1.05E-03 & 2.0580    & 1/1600 & 1.71E-03 & 1.9341       & 1/1000 & 4.73E-04 & 3.0370 &1/800  &6.09E-04&4.4027\\
        1/3200 & 2.52E-04 & 2.0586    & 1/3200 & 4.32E-04 & 1.9872         & 1/1960 & 5.78E-05 & 3.1243  &1/1600&3.02E-05&4.3339\\
        1/6400 & 6.37E-05 & 1.9840     &  1/6400 & 9.81E-05 & 2.1391        & 1/3240 & 1.31E-05 & 2.9571  &1/2000 &1.25E-05&3.9711\\
        \bottomrule
    \end{tabular}
\end{adjustbox}
\label{tab1}
\end{table}

Figure \ref{fig:loglog-wulff} presents the convergence behavior of the ZJB scheme and our proposed second-order temporal schemes at equilibrium. As anticipated, the second-order schemes, namely the PC-ZJB scheme \eqref{PC} and the BDF2-ZJB scheme \eqref{BDFk}, exhibit a quadratic convergence rate. In contrast, the ZJB scheme \eqref{eqn:dis2d} exhibits only first-order convergence with respect to the time step $\tau$. Table \ref{tab1} summarizes the convergence order tests for the PC-ZJB scheme \eqref{PC} and the BDF$k$-ZJB schemes ($k=2,3,4$) \eqref{BDFk} at equilibrium. The results confirm that the PC-ZJB and BDF2-ZJB schemes achieve second-order accuracy, while the BDF3-ZJB and BDF4-ZJB schemes achieve third-order and fourth-order accuracy, respectively.

\subsection{Time evolution of some geometric quantities}
In this subsection, we present numerical results regarding the temporal evolution of the total free energy, area, contact angles, and mesh ratio indicator using the proposed high-order temporal schemes. First, the normalized energy $W(t)/W(0)$, the relative area loss $\Delta A(t)/A(0)$, the left scaled contact angle $\theta_{d}^{l}(t)/\pi$ (only the left contact angle is considered due to symmetry), and the mesh ratio indicator $\Psi(t)$ defined in \eqref{116} are introduced to examine the trends of energy, area, contact angle, and mesh quality, respectively. These quantities are defined as follows:
\begin{align*}
W(t)/W(0)\mid_{t=t_{n}}:=W^{n}/W^{0}&,\quad\Delta A(t)/A(0)\mid_{t=t_{n}}:=\frac{A^{n}-A^{0}}{A^0},\\
\theta_{d}^{l}(t)/\pi\mid_{t=t_{n}}:&=\arccos(\partial_{s}x^{n}\mid_{s=0})/\pi,\\
\Psi(t)\mid_{t=t_{n}}=\Psi^{n}:&=\frac{\max_{1\leq i\leq N}|\bm{h}_{i}^{n}|}{\min_{1\leq i\leq N}|\bm{h}_{i}^{n}|}.
\end{align*}

We show the time evolution of the normalized energy $W(t)/W(0)$ and the relative area loss $\Delta A(t)/A(0)$ for the PC-ZJB scheme \eqref{PC} in Figure \ref{fig:energy_area_pc_new}(a1)--(a3) and (b1)--(b3), respectively. As shown in Figure \ref{fig:energy_area_pc_new}(a1)--(a3), the normalized energy $W(t)/W(0)$ clearly decreases throughout the evolution, regardless of the mesh size, time step, or Young's angle. Similarly, Figure \ref{fig:energy_area_pc_new} (b1)--(b3) indicates that the area loss occurs primarily at the beginning of the evolution and stabilizes for 
$t\gg1$, again independent of the mesh size, time step, or Young's angle. This steady-state area loss can be significantly reduced by refining the time step 
$\tau$ (see Figure \ref{fig:energy_area_pc_new}(b2)).

\begin{figure}[h!]
\hspace{-7mm}
\centering
\includegraphics[width=6.5in,height=3.3in]{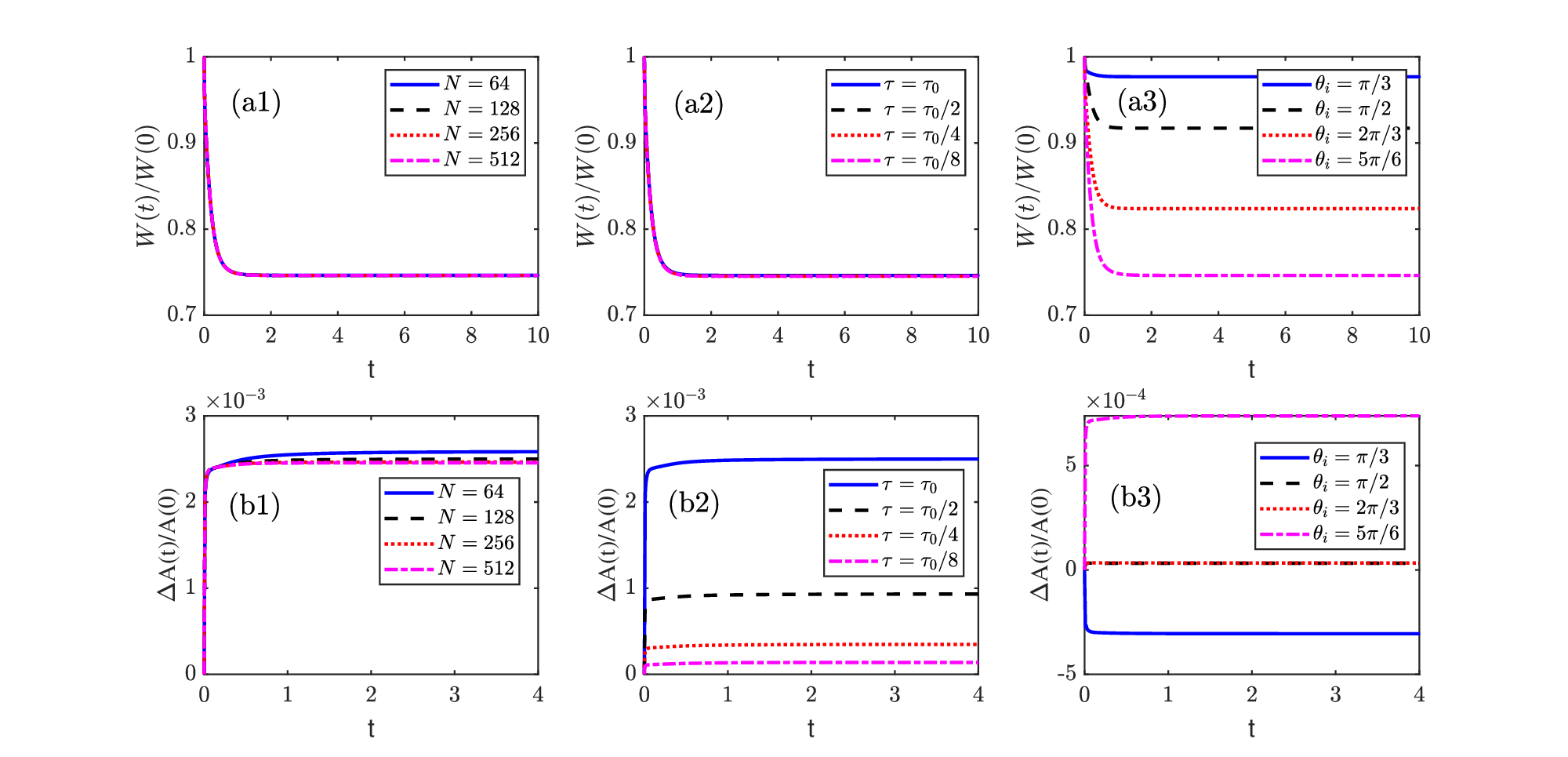}
\vspace{-3mm}
\caption{ Time evolution of the normalized energy $W(t)/W(0)$ and the relative area loss $\Delta A(t)/A(0)$ for the PC-ZJB scheme \eqref{PC}:
(a1, b1) under different mesh sizes $h=1/N$ with $\tau=0.01$ and $\theta_i=5\pi/6$;
(a2, b2) under different time steps $\tau$ with $h=1/128$, $\tau_0=0.01$, and $\theta_i=5\pi/6$;
(a3, b3) under different Young's angles $\theta_i$ with $h=1/128$ and $\tau=0.01$.
}
\label{fig:energy_area_pc_new}
\end{figure}

\begin{figure}[h!]
\hspace{-7mm}
\centering
\includegraphics[width=4.5in,height=1.8in]{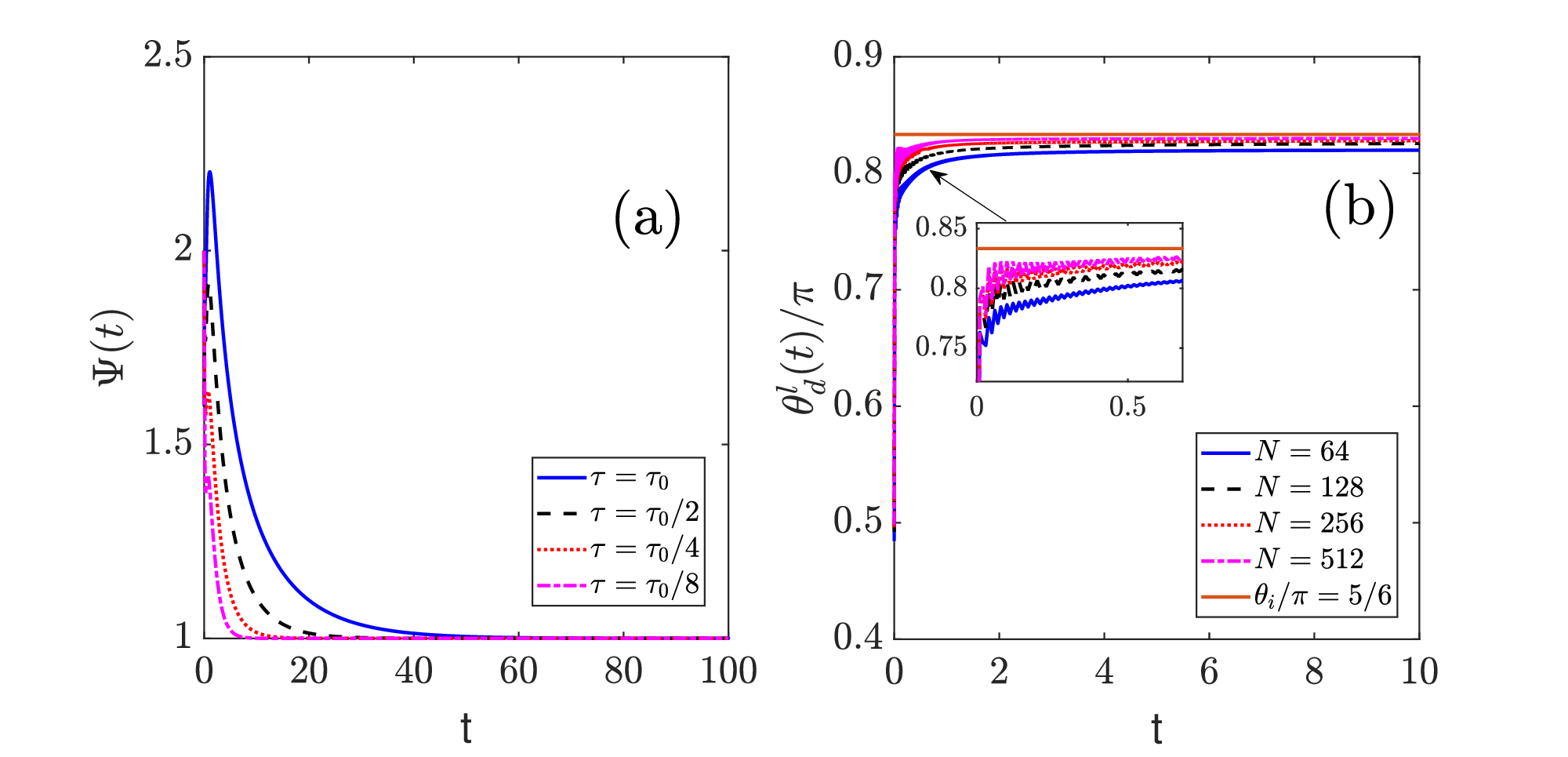}
\vspace{-3mm}
\caption{(a) Evolution of the mesh ratio indicator $\Psi(t)$ for the PC-ZJB scheme \eqref{PC} under four different time steps, with $\tau_0 = 0.01$, $N = 128$, and Young's angle $\theta_i = 5\pi/6$;
(b) evolution of the left contact angle $\theta_{d}^{l}(t)$ under four different mesh sizes, with $\tau = 0.01$ and Young's angle $\theta_i = 5\pi/6$.}
\label{fig:mesh_theta_pc}
\end{figure}

Figure \ref{fig:mesh_theta_pc} shows the temporal evolution of the  mesh ratio indicator $\Psi(t)$ under  four different time step sizes and  the left  contact angle $\theta_{d}^{l}(t)$  converging to the Young's angle  under four different mesh sizes . As clearly shown in Figure \ref{fig:mesh_theta_pc}(a), we observe that $\Psi(t)$ eventually converges to $1$, indicating asymptotic mesh equal distribution, in accordance with Proposition  \ref{prop1}, Morever, the $\Psi(t)$  can be quickly converges to $1$  when we refine the time step size $\tau$. Form Figure \ref{fig:mesh_theta_pc}(b),  we also observe the numerical convergence between the dynamic contact angle and Young's angle $\theta_i$ (red line) in the long time when we refine the mesh size  $h=1/N$ from $N = 64$ to $N = 512$.

\begin{figure}[h!]
\hspace{-7mm}
\centering
\includegraphics[width=6.5in,height=3.3in]{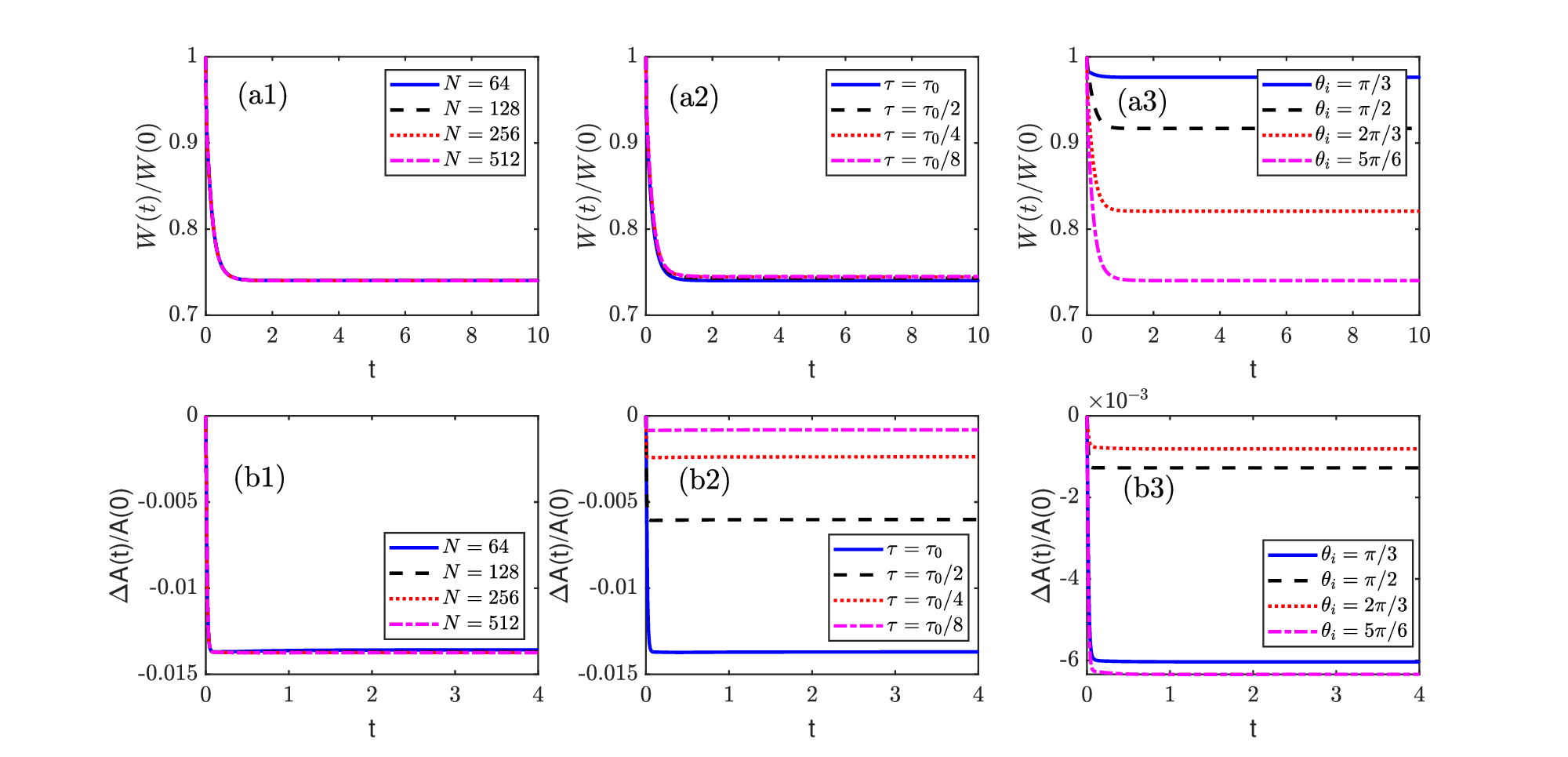}
\vspace{-3mm}
\caption{ Time evolution of the normalized energy $W(t)/W(0)$ and the relative area loss $\Delta A(t)/A(0)$ for the BDF2-ZJB scheme ($k=2$) \eqref{BDFk}:
(a1, b1) under different mesh sizes $h=1/N$ with $\tau=0.01$ and $\theta_i=5\pi/6$;
(a2, b2) under different time steps $\tau$  with $h=1/128$, $\tau_0=0.01$, and $\theta_i=5\pi/6$;
(a3, b3) under different Young's angles $\theta_i$ with $h=1/128$ and $\tau=0.01$.}
\label{fig:energy_area_2_new}
\end{figure}

\begin{figure}[h!]
\hspace{-7mm}
\centering
\includegraphics[width=4.5in,height=1.8in]{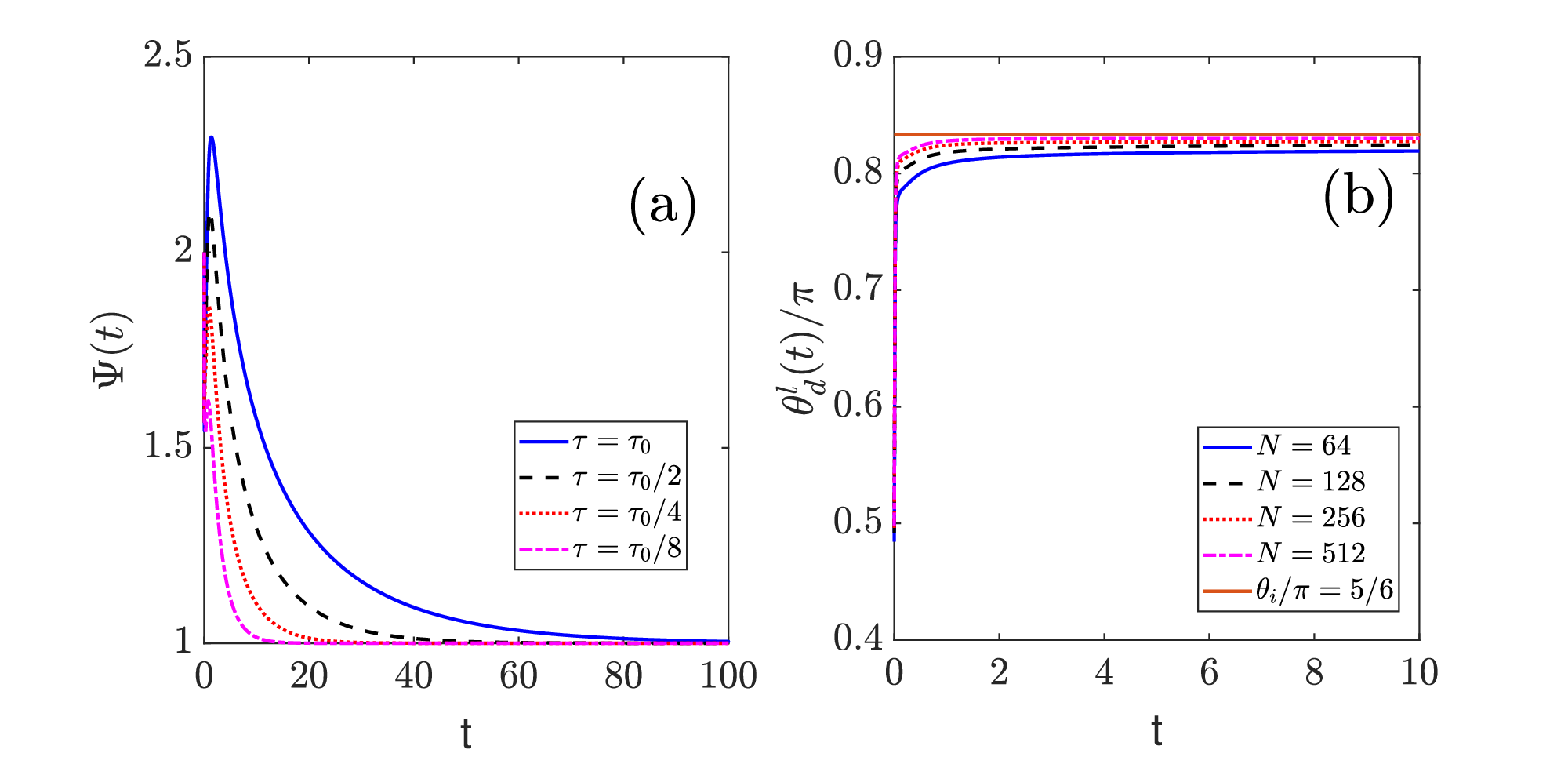}
\vspace{-3mm}
\caption{ (a) Evolution of the mesh ratio indicator $\Psi(t)$ for the BDF2-ZJB scheme ($k=2$) \eqref{BDFk} under four different time steps, with $\tau_0 = 0.01$, $N = 128$, and Young's angle $\theta_i = 5\pi/6$;
(b) evolution of the left contact angle $\theta_{d}^{l}(t)$ under four different mesh sizes, with $\tau = 0.01$ and Young's angle $\theta_i = 5\pi/6$.}
\label{fig:mesh_theta_2}
\end{figure}

\begin{figure}[h!]
\hspace{-7mm}
\centering
\includegraphics[width=6.5in,height=3.3in]{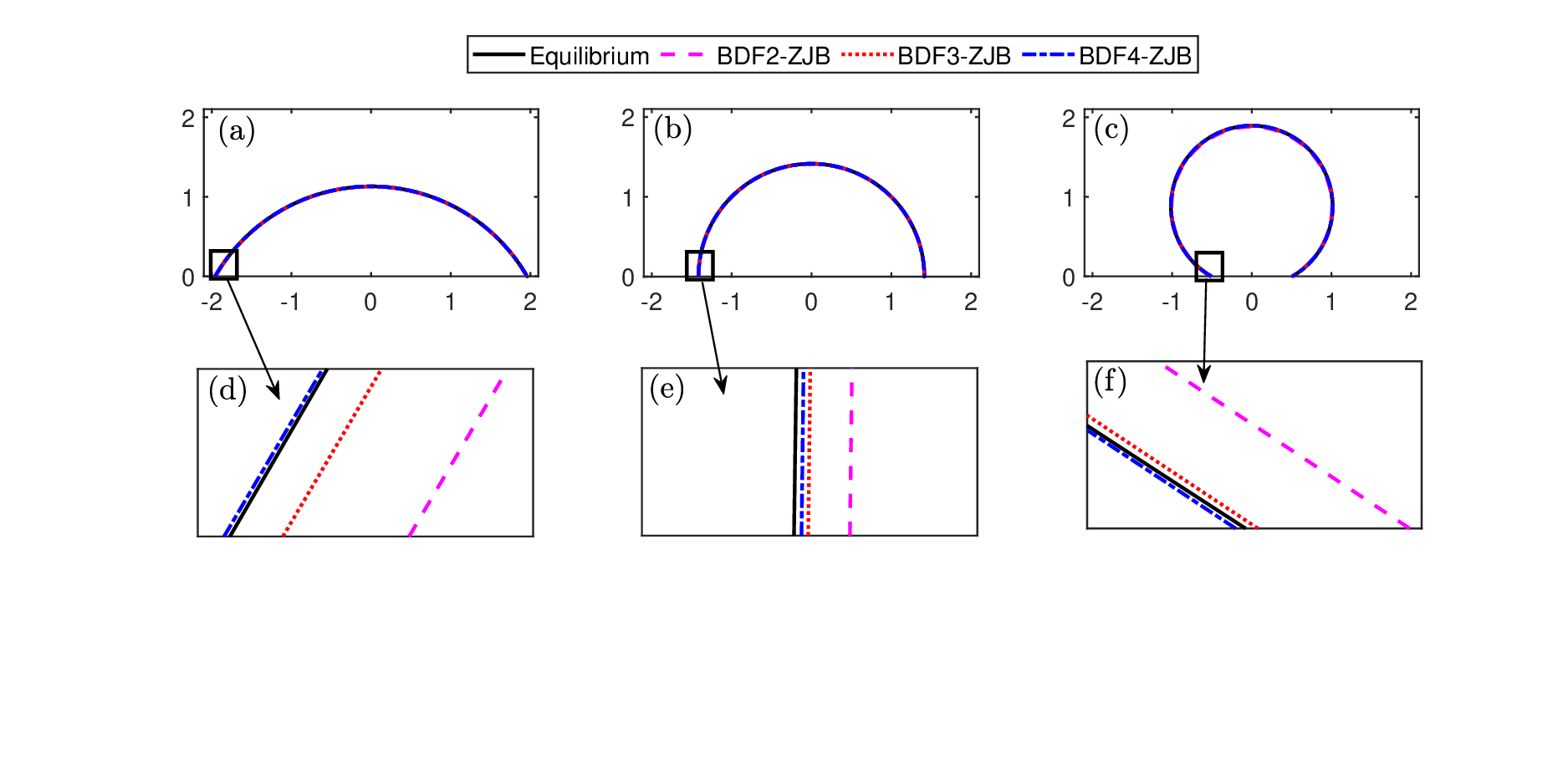}
\vspace{-25mm}
\caption{ Numerical equilibrium states obtained from the BDF$k$-ZJB scheme \eqref{BDFk} for three Young's angles: $\theta_i = \pi/3$ (a), $\theta_i = \pi/2$ (b), and $\theta_i = 5\pi/6$ (c). Corresponding zoomed-in views are shown in panels (d)--(f). Other parameters are set as $N = 128$ and $\tau = 0.01$. The solid black curves represent the theoretical equilibrium shapes constructed using \eqref{eq:Wulff}.}
\label{fig:equilibrium}
\end{figure}

We also performed numerical simulations using the BDF$k$-ZJB scheme \eqref{BDFk} with $k=2$. The results for normalized energy and relative area loss, presented in Figure \ref{fig:energy_area_2_new}, lead to the same conclusions as those from the PC-ZJB scheme \eqref{PC}. Furthermore, Figure \ref{fig:mesh_theta_2}(a) demonstrates that the scheme achieves an equal distribution of mesh points along the curve in the long time limit, consistent with Proposition \ref{prop2}. Figure \ref{fig:mesh_theta_2}(b) further shows the numerical convergence of the dynamic contact angle to Young's angle $\theta_i$ as the mesh is refined ($h=1/N$).

Numerical results of key geometric quantities from the BDF$k$-ZJB scheme \eqref{BDFk} with $k=3$ and $k=4$ support the same conclusions as those from the BDF2-ZJB scheme. Figure \ref{fig:equilibrium} compares the equilibrium states obtained using the BDF$k$-ZJB scheme with $k=2,3,4$ for three Young's angles. For these simulations, we use the same mesh size and time step size as in the previous tests ($N = 128$, $\tau = 0.01$). The results confirm that the BDF4-ZJB scheme more effectively captures the correct equilibrium contact angle than the BDF2-ZJB and BDF3-ZJB schemes.

\subsection{Evolution of complex shapes}
In this subsection, we apply the proposed high-order temporal schemes, the PC-ZJB scheme \eqref{PC} and the BDF$k$-ZJB scheme \eqref{BDFk}, to a more complex initial shape, defined as follows (see also \cite{Bao2021,Zhao2021}):
\begin{equation*}
\text{Curve I:}\left\{
\begin{aligned}
x&=(2+\cos(6\theta))\cos(\theta),~\\
y&=(2+\cos(6\theta))\sin(\theta),~
\end{aligned}
\quad\theta\in[0, \pi].
\right.
\end{equation*}

\begin{figure}[h!]
\centering
\hspace{-7mm}
\includegraphics[width=6.5in,height=2.3in]{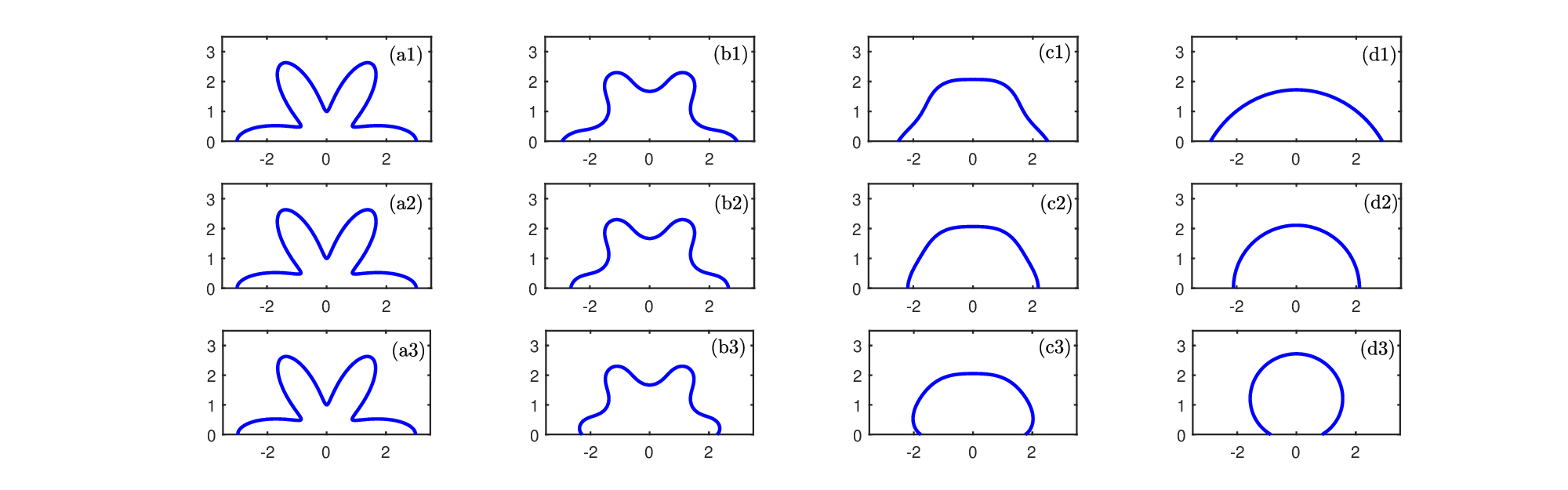}
\vspace{-3mm}
\caption{Morphological evolution of Curve I (simulated by the PC-ZJB scheme \eqref{PC}) toward equilibrium: $t = 0$; $t = 0.05$; $t = 0.1$; and $t = 5$. Results are shown for Young's angles of $\theta_i = \pi/3$ (top row), $\theta_i=\pi/2$ (middle row), and $\theta_i = 5\pi/6$ (bottom row). Other parameters: $N = 500$, $\tau = 10^{-3}$.}
\label{fig:shape-flower1}
\end{figure}

\begin{figure}[h!]
    \centering
     \begin{minipage}[b]{\textwidth}
        \includegraphics[width=\textwidth,height=0.15\textwidth]{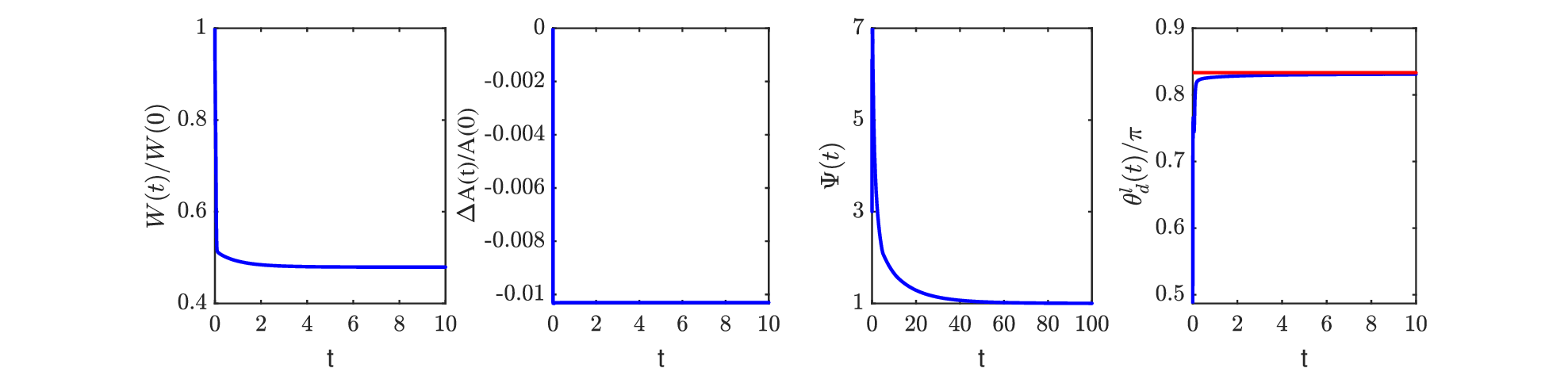}
       \centerline{(a) PC-ZJB scheme. }
    \end{minipage}
    \begin{minipage}[b]{\textwidth}
        \includegraphics[width=\textwidth,height=0.15\textwidth]{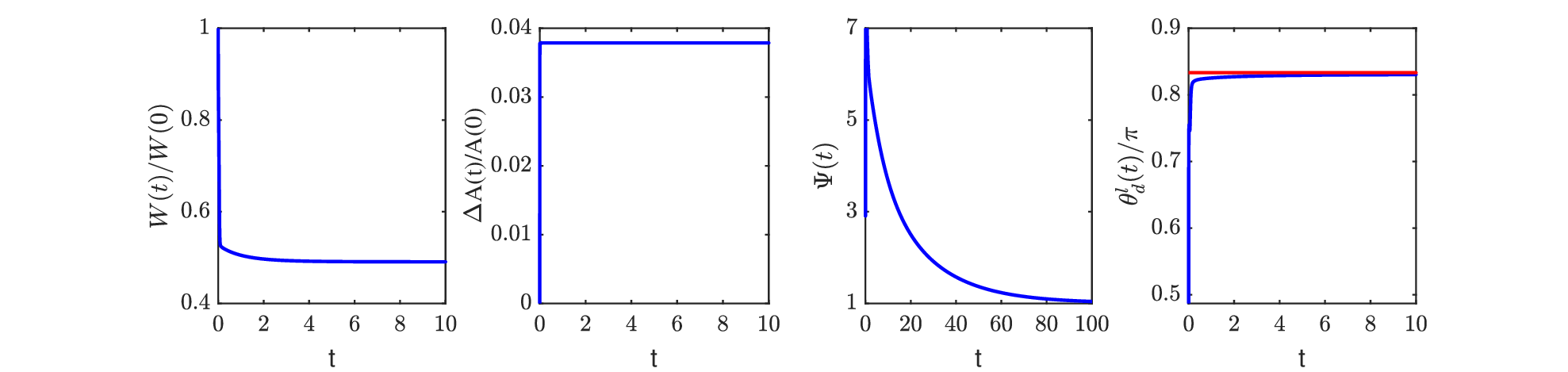}
       \centerline{(b) BDF2-ZJB scheme. }
    \end{minipage}

    \begin{minipage}[b]{\textwidth}
        \includegraphics[width=\textwidth,height=0.15\textwidth]{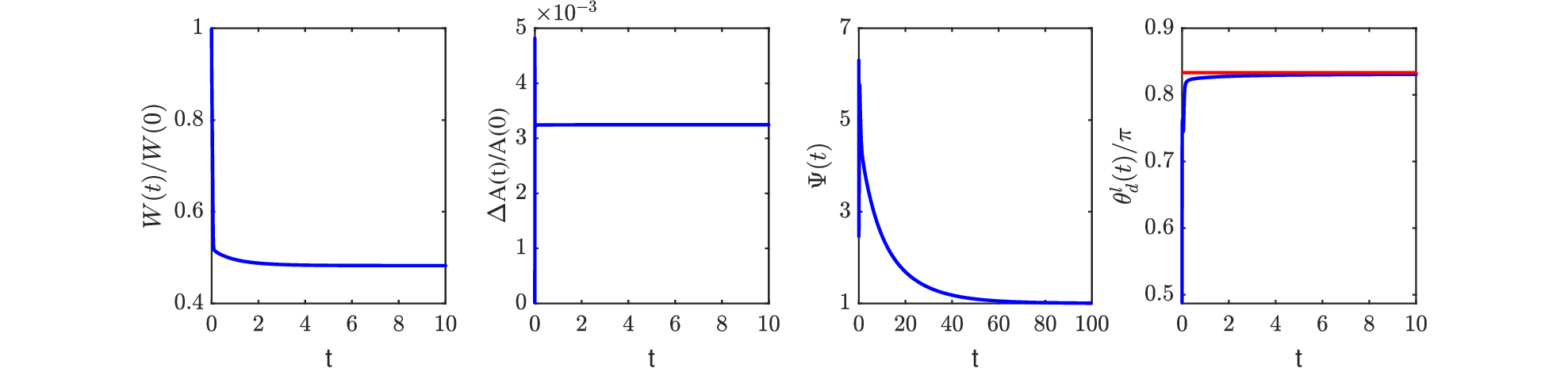}
       \centerline{(c) BDF3-ZJB scheme.}
    \end{minipage}

    \begin{minipage}[b]{\textwidth}
        \includegraphics[width=\textwidth,height=0.15\textwidth]{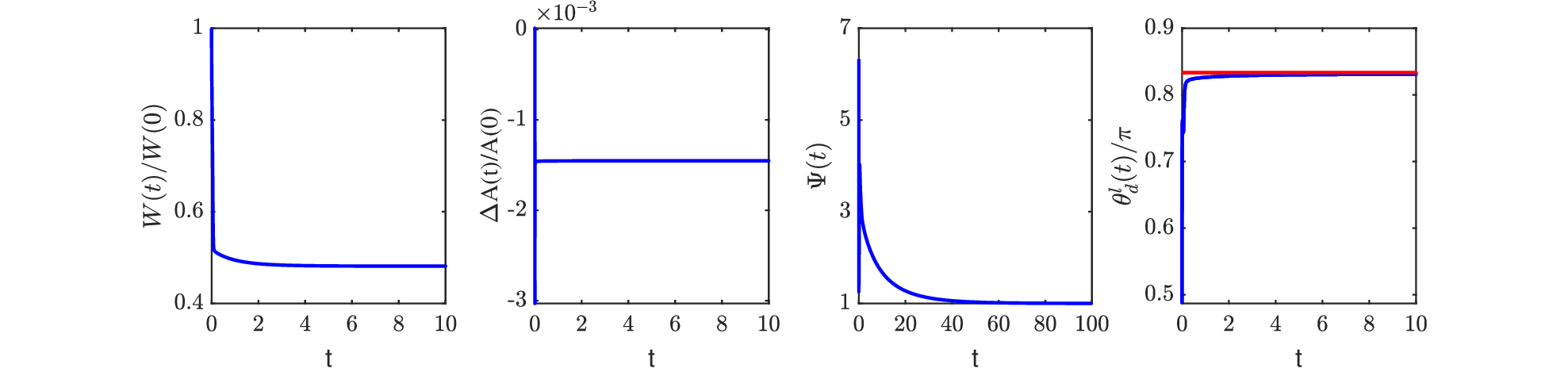}
       \centerline{(d) BDF4-ZJB scheme.}
    \end{minipage}
   \caption{Evolution of geometric quantities using high-order temporal schemes for initial Curve I. The parameters are chosen as: $N = 500$, $\tau = 10^{-3}$ and $\theta_i = 5\pi/6$.}
\label{fig:BDF-flower}
\end{figure}

Figure \ref{fig:shape-flower1} shows the morphological evolution, computed by the PC-ZJB scheme \eqref{PC}, of an island film with the initial shape Curve I toward its equilibrium state during solid-state dewetting for different Young's angles $\theta_i$. It can be observed that the petals gradually disappear, eventually forming a circular arc that contacts the substrate at angle $\theta_i$. The evolution results obtained from the BDF$k$-ZJB scheme \eqref{BDFk} for Curve I show good agreement with those presented here.

Figure \ref{fig:BDF-flower} shows the time evolution of several numerical quantities for the PC-ZJB and BDF$k$-ZJB schemes. We observe that the energy decreases and area loss occurs primarily at the beginning, along with long-term mesh equidistribution. Moreover, the last column shows that the dynamical contact angle converges to Young's angle, with the blue line representing the scaled left dynamical contact angle $\theta_{d}^l(t)/\pi$ and the red line the scaled Young's angle $\theta_{i}/\pi$.



%

\section{Conclusions}

We have proposed a class of temporal high-order (second-order to fourth-order) parametric finite element schemes for simulating solid-state dewetting in two dimensions, based on the ZJB formulation \cite{Zhao2021}. These schemess utilize a predictor-corrector strategy and backward differentiation formulae for time discretization. Furthermore, we prove that the proposed schemes are well-posed and satisfy the long-term mesh equidistribution property. Numerical results demonstrate that our schemes achieve high-order temporal convergence rates in manifold distance, in contrast to the first-order convergence of the ZJB scheme with backward Euler discretization. The development of an unconditionally energy-stable, high-order temporal scheme for long-time simulations of solid-state dewetting with isotropic surface energy remains an open challenge, which will be a primary focus of our future work.

\section*{Acknowledgement}
The authors sincerely appreciate Professor Wei Jiang for his valuable comments and suggestions.
This work was partially supported by the National Natural Science Foundation of China Grant No. 12271414.
The numerical calculations in this paper have been done on the supercomputing system in the Supercomputing Center of Wuhan University.

\section*{Data availability}
No data was used for the research described in the article.


\bigskip


\end{document}